\documentclass[11pt]{article}

\RequirePackage{fix-cm}

\usepackage{authblk}

\usepackage{endnotes}
\usepackage[section]{easy-todo}
\usepackage{array}
\usepackage{graphicx}
\let\footnote=\endnote

\usepackage{natbib}
 \bibpunct[, ]{(}{)}{,}{a}{}{,}%

\usepackage{color}

\newcommand{\npeople}{n}
\newcommand{\peopleSet}{\mathcal{N}}
\newcommand{\ngroups}{m}
\newcommand{\groupSet}{\mathcal{M}}
\newcommand{\sizeofgroups}{s}
\newcommand{\utility}{u}
\newcommand{\set}[1]{\left\{#1\right\}}
\newcommand{\paren}[1]{\left(#1\right)}
\newcommand{\pairwiseUtility}[2]{u_{#1,#2}}
\newcommand{\objweight}{\alpha}
\newcommand{\utilityAssignment}[2]{\mathsf{u}\paren{#1,#2}}
\newcommand{\uplift}{r}
\newcommand{\possibleGroups}{\mathcal{C}}
\newcommand{\possibleGroupsChar}{\mathcal{C}^{\mathcal{Q}}}
\newcommand{\possiblePartitions}{\mathcal{T}}
\newcommand{\individualMaxUtility}[1]{\mathcal{U}\paren{#1}}
\newcommand{\groupUtility}[1]{\mathsf{u}\paren{#1}}
\newcommand{\charSet}{\mathcal{Q}}
\newcommand{\possesChar}[2]{\delta_{#1,#2}}
\newcommand{\minChar}[1]{\emph{quota}_{#1}}
\newcommand{\maxChar}[1]{\emph{cap}_{#1}}
\newcommand{\RMP}[1]{RMP({#1})}

\newcolumntype{C}[1]{>{\centering\let\newline\\\arraybackslash\hspace{0pt}}m{#1}}

\usepackage{lscape}
\usepackage{afterpage}
\usepackage{longtable}

\makeatletter
\def\old@comma{,}
\catcode`\,=13
\def,{%
	\ifmmode%
	\old@comma\discretionary{}{}{}%
	\else%
	\old@comma%
	\fi%
}
\makeatother

\usepackage{amssymb}
\usepackage{amsmath}
\usepackage{amsthm} 
\newtheorem{theorem}{Theorem}
\newtheorem{proposition}{Proposition}

\usepackage{atbegshi}
\AtBeginDocument{\AtBeginShipoutNext{\AtBeginShipoutDiscard}}

\begin{document}

\title{On Finding Stable and Efficient Solutions for the Team Formation Problem}

\author[1]{Hoda Atef Yekta}\thanks{hoda.atef\_yekta@uconn.edu}
\author[2]{David Bergman\thanks{david.bergman@uconn.edu}}
\author[3]{Robert Day\thanks{robert.day@uconn.edu}}
\affil[1,2,3]{\small Operations and Information Management, University of Connecticut}

\maketitle

\begin{abstract}
The assignment of personnel to teams is a fundamental and ubiquitous managerial function, typically involving several objectives and a variety of idiosyncratic practical constraints. Despite the prevalence of this task in practice, the process is seldom approached as a precise optimization problem over the reported preferences of all agents. This is due in part to the underlying computational complexity that occurs when quadratic (i.e., intra-team interpersonal) interactions are taken into consideration, and also due to game-theoretic considerations, when those taking part in the process are self-interested agents. Variants of this fundamental decision problem arise in a number of settings, including, for example, human resources and project management, military platooning, sports-league management, ride sharing, data clustering, and in assigning students to group projects. In this paper, we study a mathematical-programming approach to ``team formation'' focused on the interplay between two of the most common objectives considered in the related literature: economic efficiency (i.e., the maximization of social welfare) and game-theoretic stability (e.g., finding a core solution when one exists). With a weighted objective across these two goals, the problem is modeled as a bi-level binary optimization problem, and transformed into a single-level, exponentially sized binary integer program. We then devise a branch-cut-and-price algorithms and demonstrate its efficacy through an extensive set of simulations, with favorable comparisons to other algorithms from the literature. 
\end{abstract}


%


\section{Introduction}
\label{sec:introduction}
Consider the task of assigning $n$ workers to $m$ teams of equal size (assume that $m$ evenly divides $n$). Each worker scores each co-worker, being asked to express an integer from 1 to (say) 10, with 10 being a most preferred teammate.

Through the application of constraints, it is easy to imagine many practical 
 variations of this basic setting. For example, perhaps there are $m$ project managers among the $n$ workers, and each team must have exactly one project manager. Each project manager is tied to a specific project plan, and workers consider the two together in their scoring. Variations with more constraints (e.g., at least one marketing person, at least two from IT, etc.) abound.

This process of putting people into teams and observing their satisfaction (or lack thereof) happens incredibly frequently in practice, yet obtaining maximum-score solutions for even modestly sized instances can be computationally prohibitive. Indeed, the maximization of the sum of the interpersonal scores among teammates is NP-Hard, and is too difficult to solve for even a few dozen workers, depending on how sophisticated of a formulation is used.

Beyond utility maximization, the economic concept of stability has played a major role in similar types of preference-reporting mechanisms. In many games, we ask if a subset of players could all improve their situation by breaking away from the game and getting together by themselves. Such a situation is called \emph{unstable}, and is considered bad for morale or for the survivability of the game as an institution. Further, it encourages the subset to collusively misstate their preferences to achieve the better outcome. (By reporting maximum scores for those in the subset, and zero for all others, group manipulation may often prove beneficial, undermining the goals of the mechanism.) Some job markets (modeled as marriage markets, such as medical-residency matching) have been shown in the classical literature to always admit stable solutions, sometimes many. Other simple markets may have no stable solution, such as in the \emph{stable roommate problem} (which can be modeled as our current problem with team-size $s=\frac{n}{m}=2$ and no constraints besides $s=2$).

This leads us to consider the problem of finding a solution that minimizes the amount that any subset (or \emph{coalition}) of players could mutually improve their situation by breaking away and forming a new valid team. We call this the \emph{maximum uplift coalition}, and consider the minimization of the maximum uplift. When the minimum is zero, the solution is stable, but even when no stable solution exists, the minimum provides a measure of instability.

A restriction to equal-sized teams is used throughout this paper; other constraints based on worker characteristics may or may not be included. Since the case when $m$ does not evenly divide $n$ but team sizes as close as possible to uniform are required can be modeled by adding dummy workers and a constraint enforcing that at most one dummy worker is assigned per team, we assume $m$ evenly divides $n$ throughout for simplicity. This (nearly-)equal-size teams assumption is crucial, driving the interesting behavior of this type of system (in contrast to hedonic games). With equal-size teams, individual actions have less influence, making group uplift a more interesting consideration than unilateral deviation. 

After a broad review of the literature (\S \ref{sec:litrev}), we provide a descriptive model (\S \ref{sec:tfp}) and then refine to a more usable single-level reformulation (\S \ref{sec:singleLevel}). We provide and then compare three algorithmic implementations (\S \ref{sec:BCP} and \S \ref{sec:experimentaleval}) before comparing to benchmark heuristics from the literature (\S \ref{sec:BenchComp}) across four distinct preference models.

\section{Related Literature}
\label{sec:litrev}
\subsection{Team Dynamics}
The study of interpersonal dynamics in teams has a huge stream of literature within the management and organizational psychology community. Though too deep to expound upon broadly here, \cite{gardner2017understanding}, for example, have considered formal models with quadratic objective terms based on pairwise utility realized among those individuals placed in the same group. See \cite{mathieu2015team} for a survey of the dynamics of people working in teams and their influence on team formation. One perspective is that a team planner will often have several practical constraints on acceptable teams, for example that each team must include a minimum number of people from a certain gender, ethnic group, or having a specific skill set \citep{Camp93}. Moreover, several models and algorithms have been developed to solve team partitioning problems that consider different forms of skill constraints \citep{FarSorHasHam11,GUTIERREZ2016150,CHEN20123888,grouping-students-in-educational-settings}. These provide support for our constraint-based architecture, with each team needing to satisfy quotas or caps on workers of particular skill-sets or other generic binary characteristics. 

\subsection{Graph Theory}
 Maximizing only the intra-group efficiency for teams of equal size is equivalent to the NP-hard \emph{balanced $k$-clique partitioning problem} (BCPP) \citep{BHASKER19911}, which has been used, for example, for sports tournament scheduling and league realignment \citep{Recalde2016}. This stream of literature focuses only the efficiency of the solution, with no consideration of stability.

\subsection{Matching}
With preferences submitted by self-interested agents, the current work is related to an extensive literature on ``matching markets'' which has become an active area of research with wide-spread applications. This work builds on the classical \emph{stable marriage problem} of \cite{gale1962college} \citep[see also][]{Gusfield:1989:SMP:68392,IRVING1985577,IwaMiy08} in which disjoint sets of men and women each rank members of the opposite gender (possibly ranking ``being alone'' above some potential mates). A \emph{stable} matching is often desired, in which no unmatched couple prefers each other over their current matching. Applications of bipartite marriage-type matching have flourished in recent years, with prominent successes in the National Residency Matching Program \citep{roth1996national}, school-choice programs \citep{pais2008school}, and kidney exchange \citep{roth2004kidney}, and even recognition of the discipline with the 2012 Nobel Prize in Economic Sciences. The tension between efficiency and stability (studied here) has been present in this stream of research, particularly in discussions of school-choice program implementations \citep{erdil2008s}. 

\subsubsection{Roommate Problem Variations}
As mentioned in Section~\ref{sec:introduction}, when the market does not consist of two disjoint classes of agents (e.g., men and women, workers and jobs, students and schools, etc.) but is instead drawn from a single pool of agents, a stable solution may not exist, even with ``teams'' of size two and no further constraints, i.e., the \emph{roommate problem} \citep{IRVING1985577}. Indeed, team formation has already been discussed as a natural generalization of the roommate problem \citep{BIRO201674}, but few have attempted to tackle all computational difficulties directly.

While a particular instance of the ($s=2$) roommate problem may not admit a stable solution, a polynomial-time algorithm \citep{DBLP:journals/dam/Irving94} can check for the existence of a stable solution and find one, if it exists. \cite{Prosser2014} present an alternative constraint-programming algorithm for this problem with partially-defined preference lists. Partially-stable outcomes in \emph{unsolvable} roommate problems (problems without stable matchings) have been defined; for example, an \emph{almost stable matching} consists of a Pareto-optimal matching with a minimum number of blocking pairs; a \emph{maximum-internally stable matching} is a solution with a maximal set of stable pairs; and a \emph{maximum irreversible matching} maximizes the \emph{number} of stable pairs \citep{Abraham2006,Tan1990,BIRO201674}. Also, \cite{van2016matching} recently explored an algorithm for the roommate problem in which people who are mutual favorites, i.e., \emph{soulmates}, are matched first, with the process then iterated.  To the best of knowledge of the authors, no formal study of determining whether or not a stable assignment exists (and producing one if it does exist) for $\sizeofgroups>2$ has previously been studied. Other generalizations of stable roommate problems include those by \cite{Cechlarova:2005:GSR:1077464.1077474}, where an agent may participate in more than one 2-person relationship.
Recently, \cite{wolfson2017fairness} study ride-sharing as an application of the roommate problem, proposing a heuristic algorithm for both efficiency and stability, but again studying only the special case of our current setting under $\sizeofgroups=2$.

It is worth noting here that much of the roommate and matching literature focuses on ordinal-preference elicitation (i.e., submission of ranked lists) while we use cardinal-preference elicitation (i.e., submission of numerical scores) for practical reasons. Because every cardinal submission can be transformed uniquely to a weak ordinal preference, we ignore this distinction for the remainder of the paper. 

\subsubsection{Hedonic Games}
This paper focuses on team formation with a restriction to \emph{equal-sized} teams, unlike \emph{hedonic games} \citep{Aziz:2011:SPA:2030470.2030497}. Negative preferences in hedonic games can result in agents not matched to teams, where here an equal-size constraint prohibits this. We thus normalize preferences to be nonnegative. Also of note is that a variety of stability concepts have been studied for hedonic games, including \emph{Nash stability}, \emph{individual stability}, and \emph{contractual individual stability}, respectively considering the benefit of leaving a current team and joining another team for each person of the market, for each person and the team receiving a defector, and for each person, the receiving team, and the team that was left \citep{Aziz:2011:SPA:2030470.2030497}.

Research has been done on the complexity of verification, existence, and calculation of solutions satisfying these notions of stability in hedonic games \citep{sung2010computational}, but the emphasis on individual deviations does not shed much light  on the difficult computational problems explored here. The equal-size constraint forces us to consider group deviation directly. Because one individual cannot defect and join another team without implications for other teams, it is important to consider improvements resulting from several trades taking place at once. Our notion of \emph{team uplift} considers the whole group of agents' preferences in forming a hypothetical alternative team. We consider individual deviation only in the appendix, showing that it seems to be less interesting with equal-size teams. 

\subsection{Team Formation: Benchmarks from the Literature}
\label{sec:litbench}
The most closely related work to our own is that of \cite{DBLP:journals/corr/WrightV15} who also explore mechanisms for team formation computationally. In the Harvard Business School Draft (\textbf{DRAFT}), agents are randomly ordered with the first $m$ agents selected as captains. Over $m-1$ rounds, each captain in turn selects her most preferred unassigned member to join her team, with the order reversing in even and odd numbered rounds. In the One-Player-One-Pick (\textbf{OPOP}) mechanism, all agents are randomly ordered with the first $m$ captains starting a team with only one pick from among the remaining non-captains. Then, each remaining non-captain in the ordering chooses her favorite team based on expected utility. If her team has another available spot, she also chooses the next person to join her team. \cite{DBLP:journals/corr/WrightV15} provide evidence in favor of these two mechanisms, making them the best benchmarks available for direct comparison in our computational experiments.

\section{The Team Formation Problem}
\label{sec:tfp}

Let $\peopleSet = \set{1, \ldots, \npeople}$ be a set of $\npeople$ agents with $\npeople=\ngroups \cdot \sizeofgroups$ and $\npeople,\ngroups,\sizeofgroups \in \mathbb{Z}^+$.  
For each $i\neq j \in \peopleSet$, let $\pairwiseUtility{i}{j} \in \mathbb{Z}^+$ represent the \emph{pairwise utility} of $i$ for being teamed with $j$, normalized to be non-negative and integer by affine transformation, with all $\pairwiseUtility{i}{i}=0$.
A feasible \emph{team formation} is a partition of $\peopleSet$ into $\ngroups$ teams $\groupSet = \set{1,\ldots,\ngroups}$, denoted by $\mathsf{t}: \peopleSet \rightarrow \groupSet$.
Hence $\mathsf{t}(i)$ is the team agent $i$ is assigned to, and the \emph{equal-size teams} restriction requires  $\vert\set{i \in \peopleSet \vert \mathsf{t}(i)=k}\vert=s$ for all $k\in\groupSet$. 
We will use $\possiblePartitions$ to denote the set of legal team formations, with the equal-size teams restriction assumed throughout.  Furthermore, let $\mathsf{c}(\mathsf{t},i) := \set{j : \mathsf{t}(j) = \mathsf{t}(i)}$ be the set of agents assigned to the same team as $i$ in $\mathsf{t}$.

It will be convenient to search over size-$\sizeofgroups$ subsets of $\peopleSet$ without specifying an entire team formation. Thus, let $\possibleGroups := \binom{\peopleSet}{\sizeofgroups}$
be the family of all subsets of $\peopleSet$ (called \emph{coalitions}) of size $\sizeofgroups$.
For any $i$, let $\possibleGroups(i) \subseteq \possibleGroups$ be those coalitions $c$ with $i\in c$.  
For any $c\in\possibleGroups(i)$, let $\utilityAssignment{c}{i}:=\sum_{j \in c} \pairwiseUtility{i}{j}$ signify $i$'s total individual utility as part of $c$.
As additional shorthand, we have $\utilityAssignment{\mathsf{t}}{i}:=
\utilityAssignment{\mathsf{c}(\mathsf{t},i)}{i}$ as $i$'s utility in a team formation and $\groupUtility{c}:=\sum_{i \in c} \sum_{j \in c} \utility_{i,j}$ as the utility of an entire coalition $c$.
We define each agent $i$'s maximum realizable utility $\individualMaxUtility{i}:=\max_{c \in \possibleGroups(i)} \utilityAssignment{c}{i}$, which can be found by simply taking the top $\sizeofgroups -1$ values of $\pairwiseUtility{i}{j}$ (until later when other constraints beyond equal-size are added).

As motivated above, we are interested in team formation with both total utility (efficiency) and stability as objectives.
Thus, our formal optimization problem, the \emph{team formation problem} (TFP) is formulated with both objectives, weighted by a scalar $\objweight \in [0,1]$ .  The first component (weighted by $\objweight$) seeks to maximize the sum of the individual utilities.
The second component (weighted by $1-\objweight$) seeks to minimize the maximum uplift $\uplift$, defined for any fixed $\mathsf{t} \in \possiblePartitions$ by a maximum of 
\[
\uplift(c,\mathsf{t}) := \sum_{i \in c} \paren{ \utilityAssignment{c}{i} - \utilityAssignment{\mathsf{t}}{i}},
\]
over all coalitions $c \in \possibleGroups$ with $\utilityAssignment{c}{i} \geq \utilityAssignment{\mathsf{t}}{i}$ for all $i \in c$. The TFP is thus modeled as:
\begin{align*}
\label{optmod:bilevel}
\tag{TFP}
&\text{max}_{\; \mathsf{t} \in \possiblePartitions}  
&  & \objweight \cdot \sum_{i \in \peopleSet} \utilityAssignment{\mathsf{t}}{i} - 
\paren{1-\objweight} \cdot \uplift
\\
& \text{s.t.} 
&& 
\uplift = \max_{c \in \possibleGroups} \set{\uplift(c,\mathsf{t}) : 
\paren{i \in c} 
\rightarrow 
\paren{ \utilityAssignment{c}{i} \geq \utilityAssignment{\mathsf{t}}{i}} }.
\end{align*}
The upper-level (leader) optimization model selects the partition and the lower-level (follower) optimization model calculates the maximum uplift coalition for the partition identified in the leader model.  The condition in the follower problem enforces that each individual in a maximum uplift coalition is not worse off.

Note that if $\objweight = 1$ we arrive at the BCPP, and if $\objweight = 0$ then stability (as measured by maximum uplift) is the only measure of interest. In the latter case, if the optimal solution is 0, then we have a fully stable solution; if positive, we have a solution that minimizes the maximum uplift.

Optimization model~(\ref{optmod:bilevel}) can be cast as a bi-level binary optimization problem by introducing a binary variable $x_{i,k}$ indicating if each agent $i$ is placed in team~$k$.  We also associate, in the follower, variables $y_i$, to indicate if person $i$ is selected in the maximum uplift coalition:
\begin{align*}
\label{optmod:bilevel-ip}
\tag{BL}
&\text{max} 
&  & \objweight \cdot \sum_{i \in \peopleSet} \sum_{j \in \peopleSet} 
			\sum_{k \in \groupSet} \pairwiseUtility{i}{j} x_{i,k} \cdot x_{j,k}
				 - 
\paren{1-\objweight} \cdot \uplift
\\
& \text{s.t.} 
&& \sum_{i \in \peopleSet} x_{i,k} = \sizeofgroups, 
&& \forall k \in \groupSet
\\
&&&
\sum_{k \in \groupSet} x_{i,k} = 1,
&&
\forall i \in \peopleSet
\\
&&&
x_{i,k} \in \set{0,1}, 
&& 
\forall i \in \peopleSet, \forall k \in \groupSet
\\
&&& 
\uplift = \max \Bigl\{ \sum_{i \in \peopleSet}\sum_{j \in \peopleSet} 
					\pairwiseUtility{i}{j} y_i y_j - 
					\sum_{i \in \peopleSet}  
					\sum_{j \in \peopleSet}
					\sum_{k \in \groupSet} y_i \pairwiseUtility{i}{j} x_{i,k}	x_{j,k}				
					 \Bigr. \\
&&& \hspace{2ex} \text{s.t.} 
					\quad 
					0 \leq \paren{1-y_i} \cdot \individualMaxUtility{i} + \sum_{j \in \peopleSet} 
					\pairwiseUtility{i}{j} y_i y_j -  
					\sum_{j \in \peopleSet}
					\sum_{k \in \groupSet} y_i \pairwiseUtility{i}{j} x_{i,k} x_{j,k}, 
&& \forall i \in \peopleSet \\
&&& \hspace{2ex} \sum_{i \in \peopleSet} y_i = \sizeofgroups, \\
&&& \hspace{2ex} y_i \in \set{0,1}, 
&&  \Bigl. \forall i \in \peopleSet  \Bigr\}.
\end{align*}

\begin{proposition}
	Model~(\ref{optmod:bilevel-ip}) is a valid formulation for the TFP.
\end{proposition}

\begin{proof}
Leader problem constraints ensure a team formation, with total utility weighted by $\objweight$ in the objective. We need only show that the follower problem identifies the maximum uplift group. 

Fix a solution to the leader problem $x'$.  Let $y'$ be a feasible solution to the follower.  By the constraint $\sum_{i \in \peopleSet} y_i = \sizeofgroups$, a properly sized coalition is identified. The first term in the follower objective function, $\sum_{i \in \peopleSet}\sum_{j \in \peopleSet} 
\pairwiseUtility{i}{j} y_i y_j$, is the total coalitional utility over every $i$ with $y'_i = 1$.  The second term evaluates to $\sum_{i \in \peopleSet}  
\sum_{j \in \peopleSet}
\sum_{k \in \groupSet} y'_i \pairwiseUtility{i}{j} x'_{i,k}x'_{j,k}$.  Therefore, for a fixed $i$, $y'_i \pairwiseUtility{i}{j} x'_{i,k}x'_{j,k}$ will evaluate to 0 if $y'_i = 0$, and to $\sum_{j \in \peopleSet}
\sum_{k \in \groupSet}  \pairwiseUtility{i}{j} x'_{i,k}x'_{j,k}$ if $y'_i = 1$.  In the latter case, this term is the utility of $i$ according to $x'$ in the leader problem.  Therefore, the $(1-\objweight)$ objective term evaluates to the total uplift of the coalition defined $y'_i$ above the $x'$ value.

Finally, the constraint	$0 \leq \paren{1-y_i} \cdot \individualMaxUtility{i} + \sum_{j \in \peopleSet} 
\pairwiseUtility{i}{j} y_i y_j -  
\sum_{j \in \peopleSet}
\sum_{k \in \groupSet} y_i \pairwiseUtility{i}{j} x_{i,k} x_{j,k}$ becomes trivially satisfied if $y'_i = 0$ (by the maximality of $\individualMaxUtility{i}$).  If $y'_i = 1$, the constraint enforces that each individual chosen must have an increase in utility to want to join an uplift coalition, completing the proof.
$\Box$
\end{proof}

\section{Single-level Reformulations}
\label{sec:singleLevel}

In order to design an efficient algorithm for solving~(\ref{optmod:bilevel-ip}) (which is bi-level, quadratically constrained and has a quadratic objective function), we transform it into a single-level linear binary optimization model with the addition of two new sets of binary variables.  For every pair $i, j \in \peopleSet$, let $w_{i,j}$ indicate if $i$ and $j$ are assigned to the same team (if $i=j$ let $w_{i,j}=0$, or simply ignore it).  Next, for every $i$ and value $v \in \set{0, 1, \ldots, \individualMaxUtility{i}}$, introduce variables $z_{i,v}$ indicating if $i$ has individual utility $v$ in the team assigned to $i$:
\begin{align*}
\label{optmod:singleLevel-PPS}
\tag{SL}
&\text{max} 
&  & \objweight \cdot \sum_{i \in \peopleSet} \sum_{j \in \peopleSet}  \pairwiseUtility{i}{j} w_{i,j}
- 
\paren{1-\objweight} \cdot \uplift
\\
& \text{s.t.} 
&& \sum_{i \in \peopleSet} x_{i,k} = \sizeofgroups, 
&& \forall k \in \groupSet
\\
&&&
\sum_{k \in \groupSet} x_{i,k} = 1,
&&
\forall i \in \peopleSet
\\
&&&
w_{i,j} \geq x_{i,k} + x_{j,k} - 1
,
&&
\forall i \in \peopleSet, \forall j \in \peopleSet,  \forall k \in \groupSet
\\
&&& 
w_{i,j} + x_{i,k} \leq x_{j,k} + 1,
&&
\forall i \in \peopleSet, \forall j \in \peopleSet, \forall k \in \groupSet
\\
&&&
\sum_{v = 0}^{\individualMaxUtility{i}} z_{i,v} = 1, 
&&
\forall i \in \peopleSet
\\
&&&
\sum_{v = 0}^{\individualMaxUtility{i}} v \cdot z_{i,v}
=
\sum_{j \in \peopleSet} \utility_{i,j} w_{i,j}, 
&&
\forall i \in \peopleSet
\\
&&&
\uplift \geq \groupUtility{c} 
-
\sum_{i \in c} \sum_{v = 0}^{\individualMaxUtility{i}} 
v \cdot z_{i,v}
-
\sum_{i \in c}
\sum_{v = \utilityAssignment{c}{i} + 1}^{\individualMaxUtility{i}}
\groupUtility{c} \cdot z_{i,v},
&&
\forall c \in \possibleGroups
\\
&&&
x_{i,k} \in \set{0,1}, 
&& 
\forall i \in \peopleSet, \forall k \in \groupSet
\\
&&&
w_{i,j} \in \set{0,1}, 
&&
\forall i \in \peopleSet, j \in \peopleSet
\\
&&&
z_{i,v} \in \set{0,1}, 
&&
\forall i \in \peopleSet, \forall v \in \set{0, 1, \ldots, \individualMaxUtility{i}}.
\end{align*}

\begin{theorem}
	Model~(\ref{optmod:singleLevel-PPS}) is a valid formulation for the TFP.
\end{theorem}

\begin{proof}
The first four sets of constraints generate a partition and link the $x$ variables with the $w$ variables. Note that the second constraint set is necessary to link the $x$ and $w$ variables, forcing there to be at least one $k$ for which $w_{i,j}$ is constrained to zero when $i$ and $j$'s teams differ.

The constraints $\sum_{v = 0}^{\individualMaxUtility{i}} z_{i,v} = 1$ and
$
\sum_{v = 0}^{\individualMaxUtility{i}} v \cdot z_{i,v}
=
\sum_{j \in \peopleSet} \utility_{i,j} w_{i,j}, 
$
define the $z$ variables, which select exactly one value from the domain of the agent's utility function and match it to the $w$ variable's generation of utility as the sum of pairwise actualizations.

The final constraint set (aside from binary constraints) links the uplift variable $\uplift$ to the uplift of each coalition.  In particular, for every coalition $c$, the first term on the right-hand side is the total utility for $c$.  The second term calculates, for each $i$ in the coalition, the individual utility realized in the team formation obtained endogenously (by $x$ and hence $w$ and $z$ variables).  The third term ensures that the constraint does not restrict the value of $\uplift$ unless it is an \emph{individually rational} uplift group for all $i \in c$. That is, if $z_{i,v} = 1$ for a value $v>\utilityAssignment{c}{i}$, then $i$ would get more utility from the endogenous team formation than from the potential coalition $c$, and so would not join $c$. An active $z_{i,v}$ variable in the third term causes the entire right-hand side to be non-positive, hence trivially satisfied; if even one agent would not join $c$ it is not considered as a stability threat.

Note that these constraints enforce $\uplift\geq 0$; for any group $c$ in the solution defined by $x$, the right-hand side evaluates to 0. With the coefficient of $\uplift$ in the objective function non-positive, $\uplift$ will take the maximum value of the right-hand sides over this constraint set. Any tight constraint from this set corresponds to a maximum uplift coalition in the solution defined by $x$.  $\Box$
\end{proof}

An alternative, exponentially sized model can be formulated, which provides a tighter linear relaxation, at the expense of model size.  This is done by creating a binary variable $t_c$  for every $c \in \possibleGroups$, indicating whether or not this coalition is part of the team formation.  
\begin{align*}
\label{optmod:singleLevel-EXP}
\tag{EXP}
&\text{max} 
&  & \objweight \cdot \sum_{c \in \possibleGroups} \groupUtility{c} t_c
- 
\paren{1-\objweight} \cdot \uplift
\\
& \text{s.t.} 
&& \sum_{c \in \possibleGroups} t_c = \ngroups, 
\\
&&&
\sum_{c \in \possibleGroups(i)} t_c = 1,
&&
\forall i \in \peopleSet
\\
&&&
w_{i,j} = 
\sum_{c \in \possibleGroups(i) \cap \possibleGroups(j) }
t_c,
&&
\forall i \in \peopleSet, j \in \peopleSet, j \neq i
\\
&&&
\sum_{v = 0}^{\individualMaxUtility{i}} z_{i,v} = 1, 
&&
\forall i \in \peopleSet
\\
&&&
\sum_{v = 0}^{\individualMaxUtility{i}} v \cdot z_{i,v}
=
\sum_{j \in \peopleSet} \utility_{i,j} w_{i,j}, 
&&
\forall i \in \peopleSet
\\
&&&
\uplift \geq \groupUtility{c} 
-
\sum_{i \in c} \sum_{v = 0}^{\individualMaxUtility{i}} 
v \cdot z_{i,v}
-
\sum_{i \in c}
\sum_{v = \utilityAssignment{c}{i} + 1}^{\individualMaxUtility{i}}
\groupUtility{c} \cdot z_{i,v},
&&
\forall c \in \possibleGroups
\\
&&&
w_{i,j} \in \set{0,1}, 
&&
\forall i \in \peopleSet, j \in \peopleSet, j \neq i
\\
&&&
z_{i,v} \in \set{0,1}, 
&&
\forall i \in \peopleSet, \forall v \in \set{0, 1, \ldots, \individualMaxUtility{i}}
\\
&&&
t_c \in \set{0,1}
&&
\forall c \in \possibleGroups
.
\end{align*}
This model replaces the $x$ variables with $t$ variables to define the team formation, and the $x$ and $w$ linking constraints with $w_{i,j} = 
\sum_{c \in \possibleGroups(i) \cap \possibleGroups(j) }
t_c$ , linking $t$ variables with $w$ variables. 

\subsection{Characteristic Constraints}
\label{sec:charCons}

As noted in \S \ref{sec:litrev} diversity of skill-sets or types is a primary concern when establishing teams.  For example, a company may want teams with individuals from diverse functional areas, with varying \emph{Myers-Briggs Type Indicator} scores \citep{MTBI}, or with different expertise.  With student teams, diversity with respect to gender, skills, or roles might be desired.   

In general we consider constraints of the form, ``each team must have a specified minimum (or maximum) number of agents with characteristic $q$.'' Formally, suppose that there are a set of characteristics $\charSet$ that each individual will either possess of not.  This is indicated by the binary parameter $\possesChar{i}{q}$ equal to 1 if and only if agent $i$ has characteristic $q$. For each characteristic $q \in \charSet$, there is a specified $\minChar{q}$ and $\maxChar{q}$, with $0 \leq \minChar{q} \leq \maxChar{q} \leq \sizeofgroups$, establishing bounds on the number of individuals possessing each characteristic that must be represented in each group. 

These bounds are easily appended to the models above. For model with $x$ variables, namely models~(\ref{optmod:bilevel-ip}) and~(\ref{optmod:singleLevel-PPS}), add:
\[
\minChar{q}
\leq 
\sum_{i \in \peopleSet}\possesChar{i}{q}x_{i,k}, 
\leq 
\maxChar{q}
\qquad
\forall q \in \charSet,
 \forall k \in \groupSet
\]
to enforce this condition. For (\ref{optmod:singleLevel-PPS}) and~(\ref{optmod:singleLevel-EXP}), with variables $t_c$, we simply refine $\possibleGroups$ to contain only those groups $c$ for which these conditions hold.  Call this set $\possibleGroupsChar$.

\section{Optimization Algorithms}
\label{sec:BCP}

The models (\ref{optmod:singleLevel-PPS}) and~(\ref{optmod:singleLevel-EXP}) are computationally challenging in practice.  Both contain, at a minimum, a pseudo-polynomial number of variables and an exponential number of constraints.  This requires the design of algorithms capable of scaling to instances of practical size.  We make use of branch-and-bound search, as is typically employed for binary optimization problems, with added routines for handling the exponentially sized portions of the model. 

This section provides a description of two optimization algorithms for TFP, one designed to solve model~(\ref{optmod:singleLevel-PPS}) and one for model~(\ref{optmod:singleLevel-EXP}).  To begin, we describe a class of optimization problems that will appear as subproblems throughout these two primary algorithms.  The \emph{cardinality-constrained binary quadratic programming problem} (CCBQP) is specified by an asymmetric, not necessarily positive semi-definite, $\nu \times \nu$
matrix $Q$, and a value $K$:
\begin{align*}
\label{optmod:ccbqp}
\tag{CCBQP}
&
\text{max}
&&
\chi^T Q \chi  \\
& \text{s.t.}
&&
\sum_{i=1}^n \chi_i = K 
\\
&&&
A\chi \geq b \\
&&&
\chi_i \in \set{0,1}, 
&&
\forall i \in \set{1, \ldots, \nu},
\end{align*}
where the additional linear constraint set $A\chi \geq b$ may be vacuous.
Recent literature has investigated effective computational models for solving various instance types of CCBQP problems.
In particular, model 2 of \cite{Lima2017} proved most efficient  in our experimental results (using commercial IP solvers) over a wide-range of instances, and so our presented results employ the following reformulation throughout. Introduce binary variables $\psi_{i,j}$ for every pair of indices $i,j \in \set{1, \ldots, \nu}$ with $i < j$, and reformulate~(\ref{optmod:ccbqp}) as follows:
\begin{align*}
\label{optmod:ccbqp2}
\tag{CCBQP*}
&
\text{max}
&&
 \sum_{i = 1}^{\nu}  Q_{i,i}\chi_i + \sum_{i = 1}^{\nu - 1} \sum_{j = i+1}^{\nu}  \paren{Q_{i,j} + Q_{j,i}} \psi_{i,j} \\
& \text{s.t.}
&&
\sum_{i=1}^\nu \chi_i = K 
\\
&&&
\psi_{i,j} \geq \chi_i + \chi_j - 1, 
&&
\forall i,j \in \set{1, \ldots, \nu}, i < j
\\
&&&
\psi_{i,j} \leq \chi_i, 
&&
\forall i,j \in \set{1, \ldots, \nu}, i < j
\\
&&&
\psi_{i,j} \leq  \chi_j , 
&&
\forall i,j \in \set{1, \ldots, \nu}, i < j
\\
&&&
\sum_{i=1}^{j-1} \psi_{i,j} + \sum_{i=j+1}^\nu \psi_{j,i} = \paren{K-1} \chi_j  , 
&&
\forall j \in \set{1, \ldots, \nu}
\\
&&&
A\chi \geq b \\
&&&
\chi_i \in \set{0,1}, 
&&
\forall i \in \set{1, \ldots, \nu}.
\end{align*}
  
We now describe three proposed algorithms for the TFP, each a branch-and-bound algorithm utilitizing a model formulated in Section~\ref{sec:singleLevel}.

\subsection{Branch and cut (\textbf{BC})}

\textbf{BC} is a branch-and-cut algorithm for solving model~(\ref{optmod:singleLevel-PPS}). All computational models for the TFP presented in this paper contain the family of uplift-defining constraints: 
\begin{equation}
\label{exp-con}
\tag{UP($c$)}
\uplift \geq \groupUtility{c} 
-
\sum_{i \in c} \sum_{v = 0}^{\individualMaxUtility{i}} 
v \cdot z_{i,v}
-
\sum_{i \in c}
\sum_{v = \utilityAssignment{c}{i} + 1}^{\individualMaxUtility{i}}
\groupUtility{c} \cdot z_{i,v},
\quad \forall c \in \possibleGroups.
\end{equation}
As opposed to adding all~\ref{exp-con} constraints at once, we propose a \emph{branch-and-cut} approach for finding constraints that might impact the optimal solution.  Namely, at each integer-search-tree node, an optimization problem is solved to identify if there exists any $c \in \possibleGroups$ for which the $c$-indexed constraint UP($c$) is violated. If such a \emph{violated constraint} exists, it is added to the model and the branch-and-bound search continues. 

\begin{proposition}
	\label{prop:exp-con}
At any integer search-tree node, for either~(\ref{optmod:singleLevel-PPS}) or~(\ref{optmod:singleLevel-EXP}), let $\uplift'$ be the value of $\uplift$ and, $\forall i \in \peopleSet$, let $v'_i$ be the unique second index for which variable $z_{i,v}$ is 1.  Let $\chi^*$ be the optimal solution to the following problem, with optimal objective value $\uplift^*$:
\begin{align*}
\label{optmodel:findViolatedCon}
\tag{VC}
&
\text{max}
&&
\sum_{i \in \peopleSet} \sum_{j \in \peopleSet} \utility_{i,j} \chi_i \chi_j - \sum_{i \in \peopleSet} v'_i \chi_i \\
& \text{s.t.}
&&
\sum_{i=1}^n \chi_i = \sizeofgroups 
\\
&&&
\sum_{j \in \peopleSet}
\utility_{i,j} \chi_j 
\geq
v'_i \cdot \chi_i
,
&&
\forall i \in \peopleSet 
\\
&&&
\chi_i \in \set{0,1}, 
&&
\forall i \in \peopleSet.
\end{align*}
There exists a violated~\ref{exp-con} constraint if and only if $\uplift^* > \uplift'$. Furthermore, if $\uplift^* > \uplift'$, and $c^* := \set{i : \chi^*_i = 1}$, then UP($c^*$) is a violated constraint.
\end{proposition}
The proof of this proposition is immediate\textemdash the mathematical program (\ref{optmodel:findViolatedCon}) finds the group $c^*$ for which each individual has a non-decreasing uplift (enforced by the constraints $\sum_{j \in \peopleSet}
\utility_{i,j} \chi_j 
\geq
v'_i \cdot \chi_i$) corresponding to the most-violated constraint. Model (\ref{optmodel:findViolatedCon}) is a special case of the CCBQP and so can be solved via model~(\ref{optmod:ccbqp2}). We note that if characteristic constraints are considered they can be directly added to the mathematical program in the statement of Proposition~\ref{prop:exp-con} by adding constraints:
\[
\minChar{q}
\leq 
\sum_{i \in \peopleSet}\possesChar{i}{q}\chi_{i}, 
\leq 
\maxChar{q},
\forall q \in \charSet.
\]

Equipped with Proposition~\ref{prop:exp-con} we can formally describe our first proposed algorithm \textbf{BC}, which solves model~(\ref{optmod:singleLevel-PPS}) by branch-and-cut. Starting with none of the constraints~\ref{exp-con}, a branch-and-bound search solves model~(\ref{optmod:singleLevel-PPS}).  At any integer-search-tree node, the optimization model in Proposition~\ref{prop:exp-con} is solved to find if there exists a violated constraint.  If one exists, the constraint UP($c^*$) is added to the model and the search continues.  Otherwise, the solution identified at the node is feasible, and a potentially improving solution.  

\subsection{Branch-cut-and-price (\textbf{BCP})}

\textbf{BCP} is a branch-cut-and-price algorithm for solving~(\ref{optmod:singleLevel-EXP}).  One simple method, which we donote by \textbf{EXP}, is enumerating all of $\possibleGroups$ (or $\possibleGroupsChar$ if $\maxChar{}$s or $\minChar{}$s are in use) and directly solving~(\ref{optmod:singleLevel-EXP}) using a state-of-the-art integer programming solver. \textbf{EXP} has the obvious shortcoming of scalability\textemdash  $|\possibleGroups|$ grows exponentially with the number of people, hence limiting the use in practical application.  

To address this shortcoming, we describe a branch-cut-and-price algorithm, \textbf{BCP}, that generates variables and constraints dynamically, as needed.  Note that not only does model~(\ref{optmod:singleLevel-EXP}) contain an exponential number of variables and an exponential number of constraints, but there is a family of constraints that contain exponentially many variables, requiring particular care in implementation.  We will interchangeably refer to the variables $t_c$ as \emph{variables} or \emph{columns}. 

\textbf{BCP} is initialized by finding any set of coalitions $\possibleGroups^0 \subseteq \possibleGroups$ representing a team formation. For any current active $\possibleGroups' \subseteq \possibleGroups$, starting with $\possibleGroups^0$, define the \emph{restricted master problem}, $\RMP{\possibleGroups'}$, as exactly formulation (\ref{optmod:singleLevel-EXP}) with $\possibleGroups$ replaced by $\possibleGroups'$, implying a restricted set of both $t_c$ columns and \ref{exp-con} constraints. 
Let $o^*(\possibleGroups')$ denote the optimal value of $\RMP{\possibleGroups'}$, with $r^*(\possibleGroups'), t^*_{c}(\possibleGroups')$, and $w^*_{i,j}(\possibleGroups')$ as the associated variables values at the particular optimal solution. To isolate the efficiency component of the objective, let $u^*(\possibleGroups'):= \frac{o^*(\possibleGroups') +  \paren{1-\objweight} r^*(\possibleGroups')}{\alpha}$.  Note that with $\possibleGroups' = \possibleGroups$, all these values correspond to the true optimal solution (of the unrestricted problem), in which case we may drop the argument. For the linear relaxation of  $\RMP{\possibleGroups'}$ (replacing each $\in \set{0,1}$ with $\in [0,1]$), we denote the optimal value and solution by adding carets (or hats)\textemdash for example, $\hat{o}^*(\possibleGroups')$ refers to the optimal value of the linear relaxation of $\RMP{\possibleGroups'}$.

With $\possibleGroups^0$ defined as a feasible team formation, $\RMP{\possibleGroups^0}$ is feasible. Yet $o^*(\possibleGroups^0)$ may not be a lower bound on $o^*$, nor will $\hat{o}^*(\possibleGroups^0)$ necessarily be an upper bound on $o^*$.  For the former, there may be additional constraints, related to $c \notin \possibleGroups^0$, that restrict $r$ to take a value higher than it does at the solution.  For the latter, there may be variables that needed to be added that are members of an improving solution.

Fix $\possibleGroups'$ and consider the linear relaxation of $\RMP{\possibleGroups'}$.  Let $\mu, \sigma_{i}$, and $\kappa_{i,j}$ be the dual multipliers of the first three listed constraints at the optimal value, defined on the appropriate indices.  Theorem~\ref{thm:BCP} establishes a condition for which one can assert that $\hat{o}^*(\possibleGroups')$ is an upper bound on $o^*$, enabling an exact branch-cut-and-price algorithm to be designed.

\begin{theorem}
\label{thm:BCP}
Let $\mathrm{rc}^*$ be the optimal value to the following problem:
\begin{align*}
\label{optmodel:pricing}
\tag{RC}
&
\text{max}
&&
\sum_{i \in \peopleSet} \sum_{j \in \peopleSet, j \neq i} \paren{\utility_{i,j} - \kappa_{i,j} } \chi_i \chi_j 
-
\sum_{i \in \peopleSet} \sigma_i \chi_i
-
\mu
\\
& \text{s.t.}
&&
\sum_{i=1}^\npeople \chi_i = \sizeofgroups 
\\
&&&
\chi_i \in \set{0,1}, 
&&
\forall i \in \set{1, \ldots, \npeople}.
\end{align*}
If $\mathrm{rc}^* \leq 0$ and the optimal value to~(\ref{optmodel:findViolatedCon}) is less than or equal to $\uplift^*(\possibleGroups')$, then $\hat{o}^*(\possibleGroups')$ is an upper bound of $o^*$.
\end{theorem}

\begin{proof}
We need only show that under the conditions of the theorem there is no set of groups $\tilde{\possibleGroups} \supset \possibleGroups'$ for which $\hat{o}(\tilde{\possibleGroups}) > \hat{o}(\possibleGroups')$.  
Let $P^*$ be the problem formed by adding to $\RMP{\possibleGroups'}$ any missing ~\ref{exp-con} constraints $\forall c \in \possibleGroups$ but keeping the index set $\possibleGroups'$ for $t_c$ columns.  Given this set of columns defined by $\possibleGroups'$, since the optimal value to $\RMP{\possibleGroups'}$ is less than or equal to $\uplift^*(\possibleGroups')$, we know that the optimal value of the LP relaxation of $\RMP{\possibleGroups'}$ is equivalent to the optimal value of the LP relaxation of $P^*$. We now show that, under the conditions of the theorem, for any coalition $\tilde{c} \notin \possibleGroups'$, the reduced cost of variable $t_{\tilde{c}}$, if it were to be added to the variables in $P^*$, is less than or equal to 0. This implies that the LP relaxation of $P^*$ is equal to $\hat{o}^*(\possibleGroups)$, which in turn implies $\hat{o}^*(\possibleGroups') = \hat{o}(\possibleGroups)$.

Fix $\tilde{c} \notin \possibleGroups'$. The reduced cost of $t_{\tilde{c}}$ is 
\[\groupUtility{\tilde{c}} - \mu - \sum_{i \in \tilde{c}} \sigma_i - \sum_{i \in \tilde{c}} \sum_{j \in \tilde{c}: j \neq i} \kappa_{i,j}
=
 \sum_{i \in \tilde{c}} \sum_{j \in \tilde{c}: j \neq i} \paren{ \utility_{i,j} - \kappa_{i,j} }
 -
 \sum_{i \in \tilde{c}} \sigma_i
 -
 \mu.
\]
By the conditions of the theorem, $\sum_{i \in \tilde{c}} \sum_{j \in \tilde{c}: j \neq i} \paren{ \utility_{i,j} - \kappa_{i,j} }
-
\sum_{i \in \tilde{c}} \sigma_i
-
\mu \leq 0$.  Since $\tilde{c}$ was arbitrarily chosen, this concludes the proof. $\Box$
\end{proof}
Note that~\ref{optmodel:pricing} is another instance of the CCBQP.

Equipped with Theorem~\ref{thm:BCP}, \textbf{BCP} proceeds as follows.   Starting from an appropriate $\possibleGroups^0$, we begin by iterating between finding improving columns, and finding violating constraints.  In particular, starting with $\possibleGroups' = \possibleGroups^0$~(\ref{optmodel:pricing}) (the \emph{pricing problem}) is iteratively solved, and whenever an improving column is found (i.e., the optimal value is greater than 0), it is added to $\possibleGroups'$.  When the optimal value drops to 0, model~(\ref{optmodel:findViolatedCon}) is solved.  If a violating constraint is found, the coalition corresponding to that constraint is added to $\possibleGroups'$, and, again, improving columns are identified.  If, however, there are no violating constraints, a global bound on the optimal value is found. One can then solve $\RMP{\possibleGroups'}$ \emph{with} integrality constraints to arrive at a feasible, potentially improving primal solution.  This is done at every search-tree node, as $\RMP{\possibleGroups'}$ can typically be solved effectively.

A branch-and-bound search is used to find the globally optimal solution.  In particular, after solving the root node of the search, by the procedure described above, a variable is chosen to \emph{branch} on. Two other nodes are created, one in which the chosen variable is forced to 0 and the other forced to 1.  The procedure continues until all nodes are pruned, whereupon the best found primal solution is the globally optimal solution to the original TFP instance. 
In the case where characteristic considerations are incorporated, one need only add constraints to the pricing problem limiting the choice of additional columns.  

More formally, a branch-and-bound search is implemented as follows.  Each search-tree node is specified by two coalitional sets:  $\possibleGroups^{\mathrm{in}}$ and $\possibleGroups^{\mathrm{ex}}$.  Coalitions  $c \in \possibleGroups^{\mathrm{in}}$ are required to be a part of the solution and $c\in \possibleGroups^{\mathrm{ex}}$ are forced to be excluded from the solution, and not be generated in subsequent pricing problems. The search-tree nodes are stored in a queue $\mathcal{L}$.  Initially, a root node $\rho$ contains $\possibleGroups^{\mathrm{in}} = \possibleGroups^{\mathrm{ex}} = \emptyset$, a relaxation value $o$ is assigned to $\rho$, and $\mathcal{L} := \set{\rho}$.  

Once processing at a search-tree node concludes, a new search node $\rho'$ is selected from $\mathcal{L}$.  The LP relaxation of~(\ref{optmod:singleLevel-EXP}) is found by solving $\RMP{\possibleGroups'}$ with added constraints enforcing each $c\in \possibleGroups^{\mathrm{in}}$ to be in the solution (i.e., $t_c = 1$) and that each $c\in \possibleGroups^{\mathrm{ex}}$ is excluded (i.e., $ t_c = 0$). This is solved by iteratively finding new columns to add to $\possibleGroups'$ via~(\ref{optmodel:pricing}) (adding constraints to exclude any $c \in \possibleGroups^{\mathrm{ex}} \cup \possibleGroups^{\mathrm{in}}$ by adding the constraint $\sum_{i \in c} \chi_i - \sum_{i \notin c} \chi_i \leq \sizeofgroups - 1 $), and then finding the most violated constraint by~(\ref{optmodel:findViolatedCon}).  By Theorem~\ref{thm:BCP}, if no improving column and no violated constraint is found, the LP relaxation is solved and a valid upper bound for the search node $o'$ is found.  

If $o'$ is less than or equal to the value of the best known feasible solution, the node is pruned and search continues.  Otherwise, the master problem $\RMP{\possibleGroups'}$ is solved with the integrality constraints included (along with constraint enforcing group include/exclusion in $\possibleGroups^{\mathrm{in}}$ and $\possibleGroups^{\mathrm{ex}}$).  If this solution is the best known, it is recorded.  

If $o'$ is still strictly larger than the value of the objective value of the best known solution, two descendant nodes are created by selecting a coalition $c' \in \possibleGroups'$ for which $0 < t_{c'} < 1$.  In one node $c'$ is added to $\possibleGroups^{\mathrm{in}}$ and in the other it is added to $\possibleGroups^{\mathrm{ex}}$. 

If $|\mathcal{L}| = 0$, the best known solution must be optimal.  

\subsection{Implementation details}
\label{sec:implementationDetails}

We implemented techniques \textbf{BC}, \textbf{EXP}, and \textbf{BCP} on an Intel(R) Core(TM) i7-4770 CPU @ 3.40GHz
and 8 GB RAM, written in C++ and compiled with GCC 4.8.4. IP models were solved with \texttt{Gurobi 7.5.1}. Unless otherwise noted, the experiments below all use an 1800 second time limit. All instances used in the experimental evaluation and MATLAB instance generators are available upon request. Further details are as follows.

\noindent \textbf{BC}: solved by \texttt{GUROBI} with cuts~(\ref{optmodel:findViolatedCon}) identified via a \texttt{callback}.   Through preliminary experimentation, the following settings for \texttt{GUROBI} were found most effective: \texttt{PreCrush} = 1; \texttt{DualReductions} = 0; \texttt{LazyConstraints} = 1. Cuts \ref{exp-con} are found by solving~(\ref{optmodel:findViolatedCon}) with the reformulation into model~(\ref{optmod:ccbqp2}).

\noindent \textbf{EXP}: solved by \texttt{GUROBI} with default setting via enumerating all of $\possibleGroups$.

\noindent \textbf{BCP}: solved by starting with an arbitrary solution (found by taking the first $\sizeofgroups$ and putting them in one group, then the next $\sizeofgroups$, etc.), and adding these groups to $\possibleGroups^0$.  We then select 100 other random groups to add to $\possibleGroups^0$ as follows.  For each of these teams, we first select a random person.  Then, having selected a group $g'$ of size $s' < \sizeofgroups$, each unselected person $j' \notin g'$ is assigned probability $p_{g'}(j') = \frac{\sum_{i \in g'} \utility_{i,j'}}{\sum_{i \in g'} \sum_{j \notin g'} \utility_{i,j}}$ and is chosen to be the next person added to $g'$ with probability $p(j')$.  This process continues until a group of size $\sizeofgroups$ is found and then that group is added to $\possibleGroups^0$.  This process is repeated 100 times.

In the branch-and-bound search, the following algorithmic specifications are set. Each pricing problem is solved by \texttt{GUROBI} with \texttt{Cuts} = 0 (this was found most effective in preliminary computational results) using reformulation~\ref{optmod:ccbqp2}.  The nodes $\mathcal{L}$ are stored in a priority queue and the node with the largest LP relaxation of the parent that created the node is selected. The group with the most fractional $t_{g'}$ is chosen to branch on.

\section{Experimental Evaluation of \textbf{EXP}, \textbf{BC}, and \textbf{BCP}}
\label{sec:experimentaleval}

Our results can be summarized as follows, with details provided in each corresponding subsection:
\begin{itemize}
\item \textbf{BCP} tends to dominate \textbf{EXP} and \textbf{BC} on the number of instances solved and optimality gap, though \textbf{BC} can at times provide an optimality bound when the others cannot. (\S \ref{sec:OptSpeedGap})
\item Under \textbf{BCP}, larger $\utility^{\mathrm{max}}$ values tend to result in more easily solved instances, despite a larger pseudo-polynomial formulation. (\S \ref{sec:umax})
\item The boundary values $\alpha \in \left\{ 0, 1\right\}$ result in inferior solutions and should be avoided in favor of, for example,  $\alpha \in \left\{ 0.01, 0.99\right\}$. A small consideration of efficiency finds more/better stable solutions in the same amount of time, or conversely, a small consideration of stability finds more stable solutions of equal efficiency with only a small amount of additional computation. (\S \ref{sec:boundary})
\item The type of characteristic constraints generated here have a signficant impact on stability but an insignificant effect on total utility. If such constraints are desired, Pareto improvments may often be available. (\S \ref{sec:CharCon})
\end{itemize}

Based on these results, the comparison of our own techniques to the benchmarks from the literature (in \S \ref{sec:BenchComp}) restricts attention to \textbf{BCP} with $\alpha \in \set{0.01, 0.5, 0.99}$. 

\subsection{Instance Generation} 
\label{sec:datagen}

To explore the multidimensional parameter space, we focused on markets of size $(\ngroups,\sizeofgroups) \in \left\{(2,4),(2,12),(3,8),(4,2),(4,4),(4,6),(4,8),(6,4),(8,4) \right\}$, including therefore instances with $\npeople \in \left\{ 8,16,24,32\right\}$. For each $(\ngroups,\sizeofgroups)$-pair we generate 20 instances (five for each of the four preference generation models described below) using $\utility^{\mathrm{max}}=25$. Further, for $(\ngroups,\sizeofgroups)=(6,4)$ we also generated twenty instances each for $\utility^{\mathrm{max}} \in \left\{ 5, 100\right\}$, for use in \S \ref{sec:umax}, a study of the effect of the $\utility^{\mathrm{max}}$ parameter. Finally for each these eleven market parameter settings (nine with $\utility^{\mathrm{max}}=25$ and one each with $\utility^{\mathrm{max}}=5,100$) we also generate a parallel set of instances with three personal characteristics (i.e., with $|\charSet|=3$). This doubles the number of instances, resulting in 440 simulated markets, 220 each for $|\charSet| \in \set{0,3}$.

To ensure feasibility when supplementing an instance with characteristic constraints, $\possesChar{i}{q}$ values are determined by solving the following optimization problem:
\begin{align*}
\label{FBmodel}
\tag{FB}
&\text{max}   
& & \sum_{i \in \peopleSet} \sum_{q \in \charSet} w_{i,q} \cdot \delta_{i,q} \\
& \text{s.t.} 
&& \minChar{q} \leq \sum_{i \in \peopleSet \vert \mathsf{t}(i)=k}\possesChar{i}{q} 
\leq \maxChar{q} \qquad \forall q \in \charSet, \forall k \in \groupSet \\
&&& \possesChar{i}{q} \in \set{0,1}
&
\end{align*}
For any values of $w_{i,q}$, this model produces $\possesChar{i}{q}$ values so that at least one valid team exists. For each instance, values $w_{i,q} \sim U[-1,1]$ are drawn independently $\forall i \in \peopleSet$ and $\forall q \in \charSet$. This provides a random assignment of all three personal binary charecteristics such that quotas and caps as defined in \S \ref{sec:charCons} are always feasible.

Proceeding, each market is defined by $\utility_{i,j}$ values, which can be generated randomly according to a preference generation model. Here we employ two models from the literature (\textbf{G1} and \textbf{G3}) and offer two new variants of our own (\textbf{G2} and \textbf{G4}).

\noindent \textbf{Monotone Common-Value (G1)}:  $\utility_{i,j} = j + \epsilon_{i,j}$, where $\epsilon_{i,j} \sim N\paren{0,\frac{\npeople}{5}}$. (Note that $\epsilon_{i,j}$ is re-sampled until $\utility_{i,j}$ is a positive number.)  This recreates the generation procedure of \cite{othman2010finding}, with a common-value component $j$ (equal to the agent index) and a normally distributed private deviation from the common baseline opinion of an agent's value. (This may simulate, for example, a sports draft where the relative order of players does not vary much from one selecting team to the next, resulting in highly correlated preferences.)

\noindent\textbf{Clustered Common-Value (G2)}: $\utility_{i,j} = l_{j} + \epsilon_{i,j}$, where the common value $l_{j}$ has four levels, $l_j \in \set {0, n/4,n/2, 3n/4}$ for four equal segments of the market. The resulting preferences are similar to \textbf{G1}, but with a heavier reliance on the private values to distinguish individuals. This models a market with less extreme public agreement on the value of agents than \textbf{G1}.

\noindent\textbf{Uniform independent preferences (G3)}:  $\utility_{i,j} \sim U\set{0,100}$, drawn independently from a discrete uniform distribution. For each agent, $\npeople-1$ numbers are chosen and sorted. Then, the differences between consecutive draws, in sorted order, are used to describe the utility of person $i$ for other agents in the market \citep{DBLP:journals/corr/WrightV15}.

\noindent\textbf{Affinity Preferences (G4)}: We designed this set of instances to model settings in which the characteristics of agents affect their utilities.
First, characteristic sets are randomly generated based on model~(\ref{FBmodel}).
We assume that the first two characteristics represent traits for which \emph{like} agents tend to have a natural affinity (e.g., age or gender) while the third characteristic is viewed favorably by all (e.g., higher skill level).
We therefore generate $\utility_{i,j}= \beta_{1,i}\gamma_{i,j,1}+\beta_{2,i}\gamma_{i,j,2}+\beta_{3,i}\delta_{j,3}+\epsilon_{i,j}$, where binary $\gamma_{i,j,q}=1$ and if and only if $\delta_{i,q}=\delta_{j,q}$  for $q \in \set{1,2}$. Further, independently drawn personal parameters $\beta_{1,i}, \beta_{2,i} \sim U[-1,2],~\beta_{3,i} \sim U[0,2]$, and $\epsilon_{i,j} \sim N(0,1)$ reflecting that affinity effects are twice as likely to attract likes values as repel them (and to varying degrees), and that characteristic 3 is always viewed favorably but to varying dergrees.

In all four data-generation schemes, an $\npeople \times \npeople$ utility matrix $U$ is generated as above and then \emph{ipsative} scores are generated by normalizing with respect to the mean and standard deviation of the utility vector of agent $i$, then rescaled and rounded to make utilities positive, integer, and no more than $\utility^{\mathrm{max}}$ \citep{baron1996strengths}.

\subsection{Optimality, Gap, and Computation Time Analysis}
\label{sec:OptSpeedGap}
We compare the relative efficiency of \textbf{EXP}, \textbf{BC}, and \textbf{BCP}, running each of the 220 unconstrained (i.e., $|\charSet| = 0$) market instances under $\objweight \in \set{0, 0.01, 0.5, 0.99, 1}$, resulting in 1100 runs of each algorithm. Detailed solution statistics appear in the Appendix, while Figure~\ref{fig:PP} depicts a cumulative distribution plot of performance.
\begin{figure}
	\centering
	\includegraphics[scale=0.4]{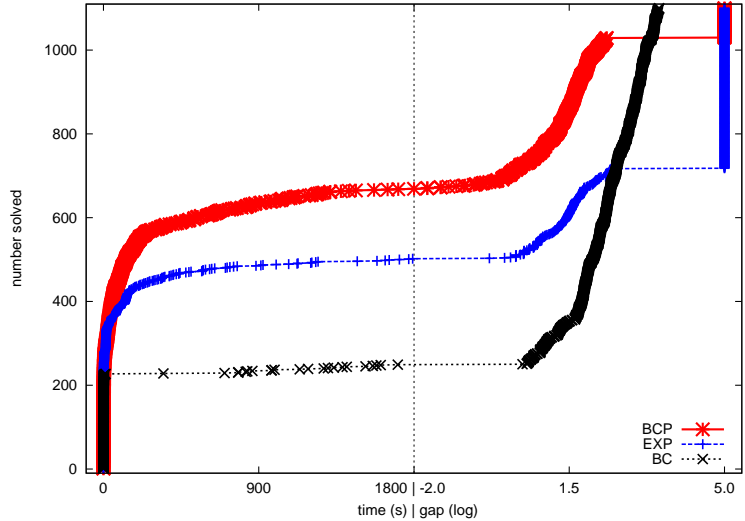}
	\caption{Cumulative distribution plot of performance, comparing the three algorithms developed in this paper.}
	\label{fig:PP}
\end{figure}
For each algorithm, the left half provides a plot with height equal to the cumulative number of instances solved at the time given on the horizontal axis. In the right half, the height of the plot corresponds to the number of remaining instances (unsolved at 1800s) with at most the log absolute gap (i.e., $\log (\textnormal{UB} - \textnormal{LB})$) shown on the horizontal axis. As a convention, we show an absolute gap of 100,000 (log of which is 5) for those instances without a definite absolute gap (i.e., for \textbf{EXP} if memory limit is hit and for \textbf{BCP} if no upper bound is proven, meaning the root node is unresolved).

Figure~\ref{fig:PP} provides clear evidence of the superiority of \textbf{BCP} over \textbf{EXP}. \textbf{BC} is the most robust, in that it can provide a gap for all instances tested, being an entirely memory controlled procedure. However, both \textbf{BCP} and \textbf{EXP} solve many more instances and leave a much smaller relative gap.

In summary, aggregated over these 1100 runs, \textbf{BCP} identifies a strictly better solution (lower bound) than the other two algorithms in 420 instances, and proves a strictly tighter relaxation bound (upper bound) in 344 instances.  In comparison, \textbf{BC} / \textbf{EXP} provide strictly best lower bound and upper bound, respectively, in 36 / 57 and 71 / 43 instances.  For those instances solved to optimality by both \textbf{BC} and \textbf{EXP}, the solution times for \textbf{EXP} are far superior and so we use \textbf{EXP} for the remaining comparisons to \textbf{BCP}.

\begin{figure}[ht] 
	\label{ fig7} 
	\begin{minipage}[b]{0.5\linewidth}
		\centering
		\includegraphics[width=1\linewidth]{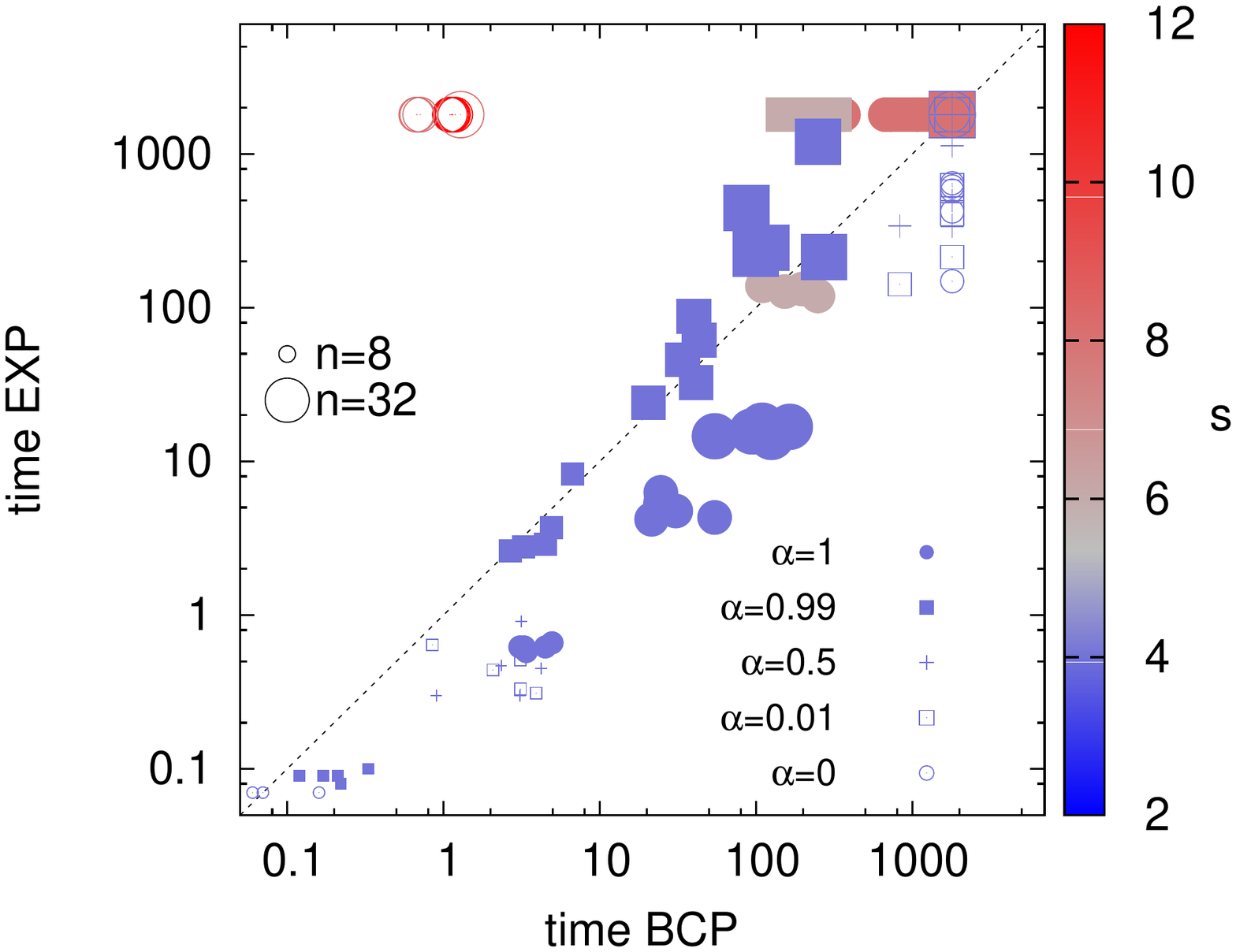} 
		\caption{Runtime for \textbf{G1}} 
		\label{fig:scatG1}
		\vspace{4ex}
	\end{minipage}
	\begin{minipage}[b]{0.5\linewidth}
		\centering
		\includegraphics[width=1\linewidth]{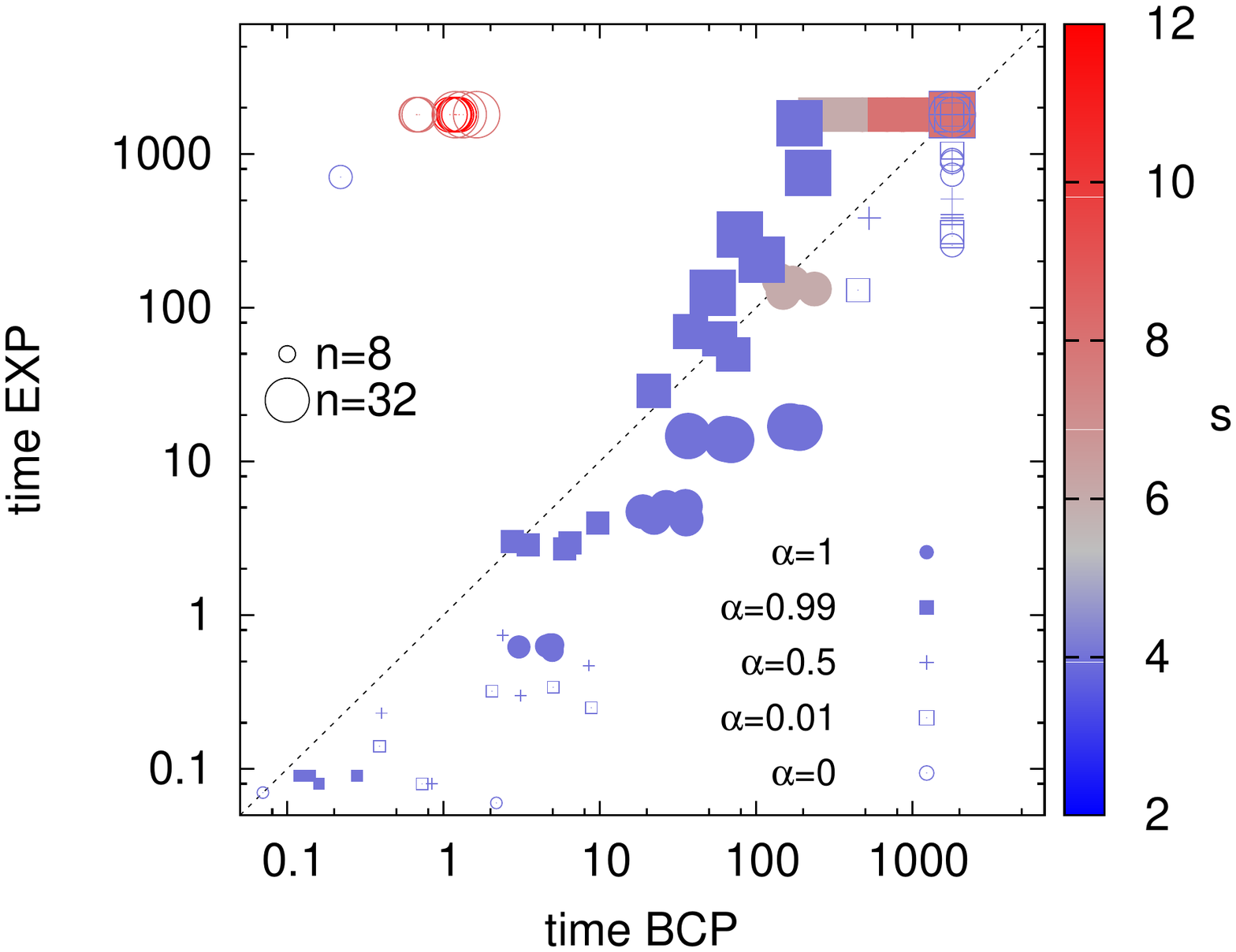} 
		\caption{Runtime for \textbf{G2}} 
		\label{fig:scatG2}
		\vspace{4ex}
	\end{minipage} 
	\begin{minipage}[b]{0.5\linewidth}
		\centering
		\includegraphics[width=1\linewidth]{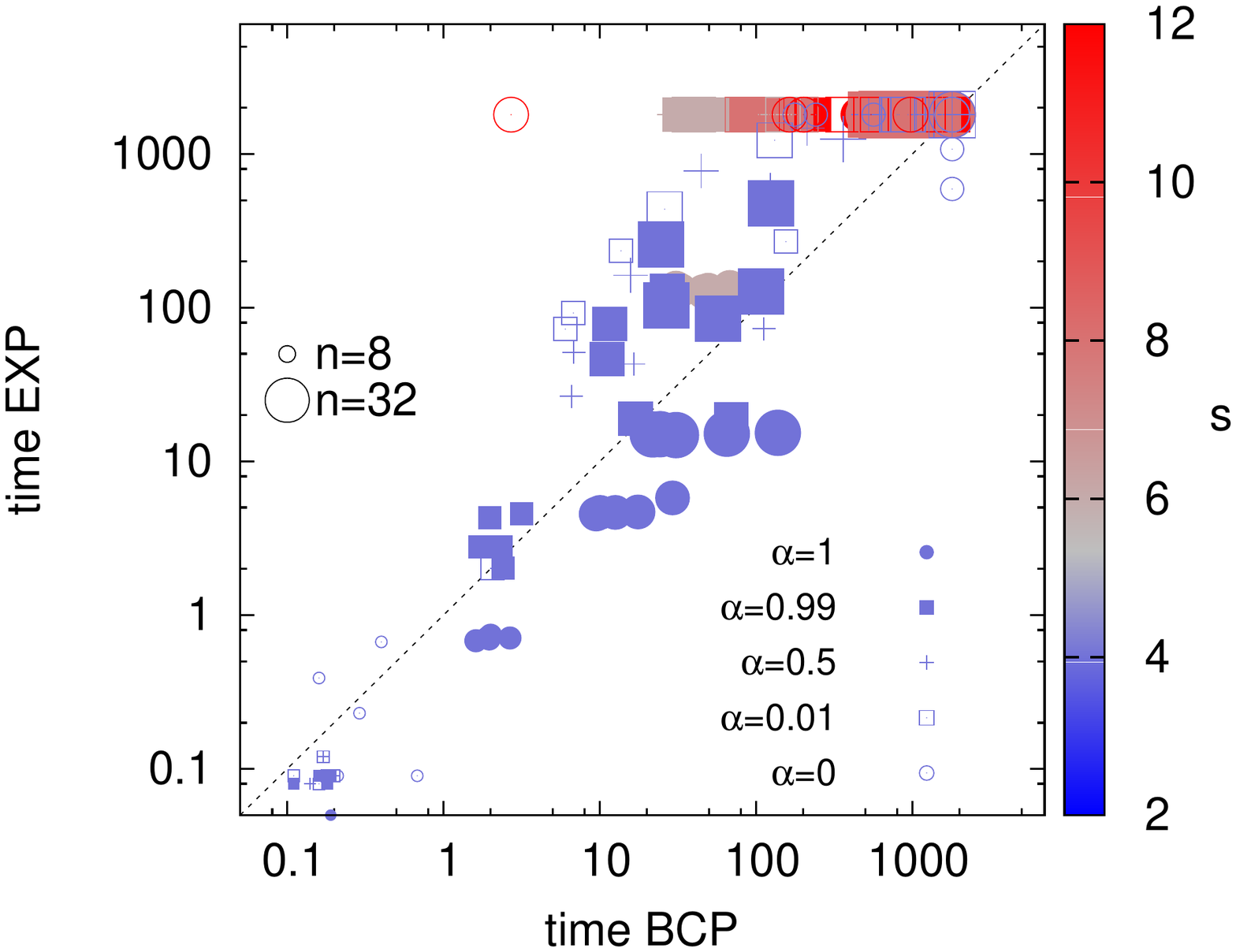} 
		\caption{Runtime for \textbf{G3}} 
		\label{fig:scatG3}
		\vspace{4ex}
	\end{minipage}
	\begin{minipage}[b]{0.5\linewidth}
		\centering
		\includegraphics[width=1\linewidth]{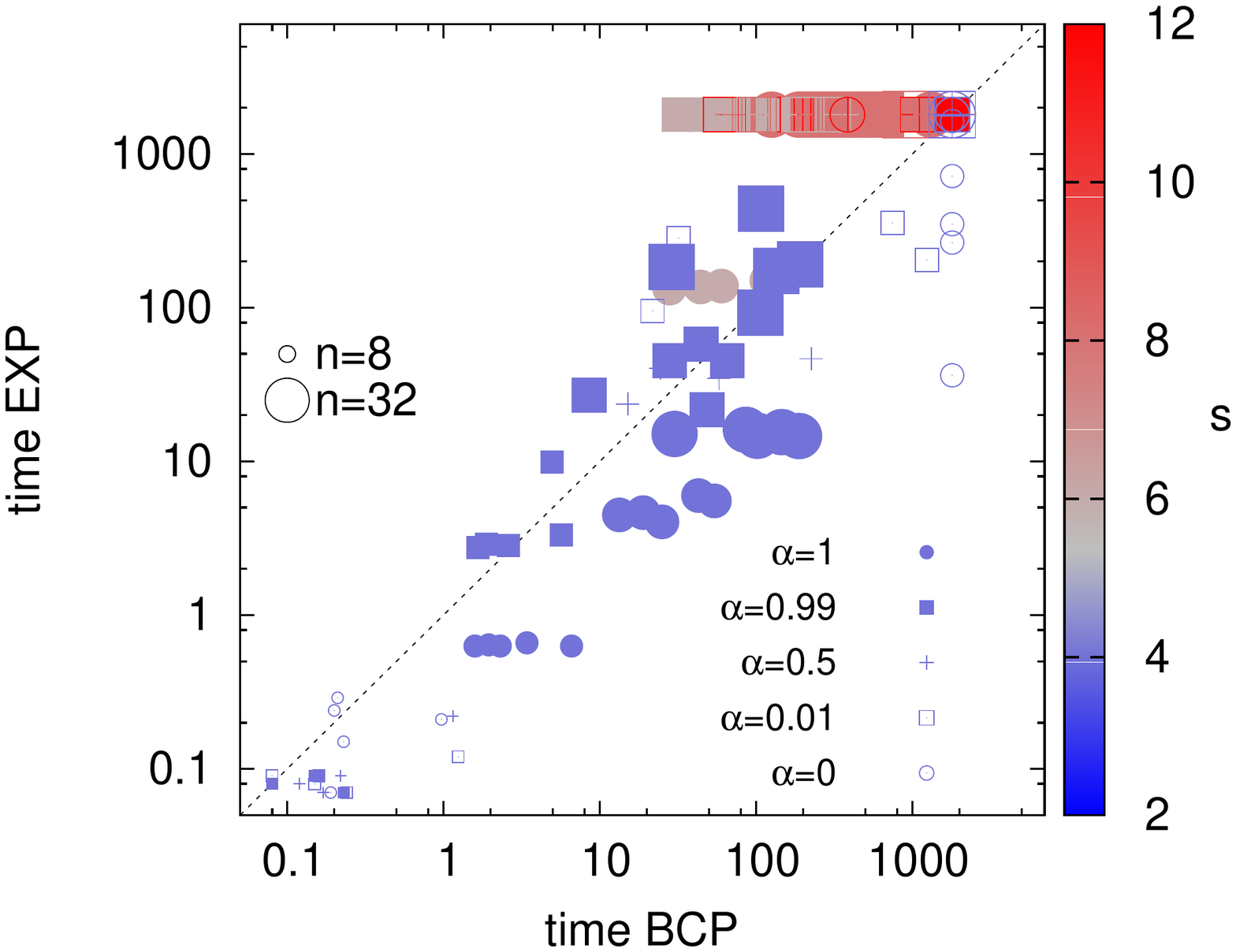} 
		\caption{Runtime for \textbf{G4}} 
		\label{fig:scatG4}
		\vspace{4ex}
	\end{minipage} 
\end{figure}

Figures~\ref{fig:scatG1} through~\ref{fig:scatG4} provide more detailed comparison of \textbf{BCP} and \textbf{EXP} through scatter plots, one for each of the four generation schemes.  In particular, for each generation scheme, a scatter plot consisting of a point per run (275 runs) with coordinates given by solution time of \textbf{BCP} and \textbf{EXP} is depicted.  The size of each point corresponds to the number of agents in the instance.  The color (gradient, from red to blue) corresponds to the team size, $\sizeofgroups$.  The point style corresponds to the value of $\alpha$, with axes depicted in logscale. 

These figures show that for the instances mutually solved, the solutions times are in general comparable, but that there are many instances unsolved by \textbf{EXP} that are solved by \textbf{BCP}, for all generations schemes.  This is more apparent for \textbf{G3} and \textbf{G4} than for \textbf{G1} and \textbf{G2}, where \textbf{BCP} is able to solve even more instances.   Table~\ref{tab:solvedInst} reports the number of instances solved by \textbf{BCP} and \textbf{EXP}, respectively.  This table suggests that \textbf{EXP} is only slightly affected by the generation scheme, but that \textbf{BCP} is able to solve significantly more instances for \textbf{G3} and \textbf{G4}, which suggests that varying degrees of common-value dependency makes instances more challenging.

\begin{table}[t!]
	\centering
	\footnotesize
	\caption{Number of instances solved in 1800 seconds by \textbf{BCP} and \textbf{EXP}}
	\label{tab:solvedInst}
	\begin{tabular}{C{15ex}|C{15ex}|C{15ex}|C{15ex}|C{15ex}}
		& \textbf{G1}  & \textbf{G2} & \textbf{G3}  & \textbf{G4}  \\ \hline
		\textbf{BCP} & 131           & 133           & 157           & 145           \\
		\textbf{EXP} & 89            & 89            & 94            & 91           
	\end{tabular}
\end{table}

The scatter plots also clearly exhibit that \textbf{EXP} is only superior to \textbf{BCP} when the number of agents and the size of the groups are relatively small. This coincides with expectations, since \textbf{EXP} is a precise model of the problem\textemdash however, as the problem size grows on any dimension, the application of \textbf{EXP} becomes prohibitive due to either memory restrictions or, even when memory limits are sufficient, resolution difficulty. 

Based on the analysis in this subsection, we use the solutions obtained by \textbf{BCP} for the subsequent analysis. Note that even if an upper bound is not proven by \textbf{BCP}, a lower bound (i.e., high-quality feasible solution) can still be obtained by solving for the best solution using the columns generated.  This contrasts with \textbf{EXP}, where, if the memory limit hits, no feasible solution will be available.  

\subsection{The effect of $\utility^{\mathrm{max}}$}
\label{sec:umax}


We next investigate the sensitivity of \textbf{BCP} to $\utility^{\mathrm{max}}$.  Figure~\ref{fig:PPMaxUtil} provides a cumulative distribution plot of performance for $\utility^{\mathrm{max}} \in \left\{ 5, 25, 100 \right\}$ on all instances generated with $m = 6$ and $s = 4$. (Four preference models, each with five instances and five values of $\alpha$, result in 100 runs for each of three $\utility^{\mathrm{max}}$ values.)  

\begin{figure}
	\centering
	\includegraphics[scale=0.4]{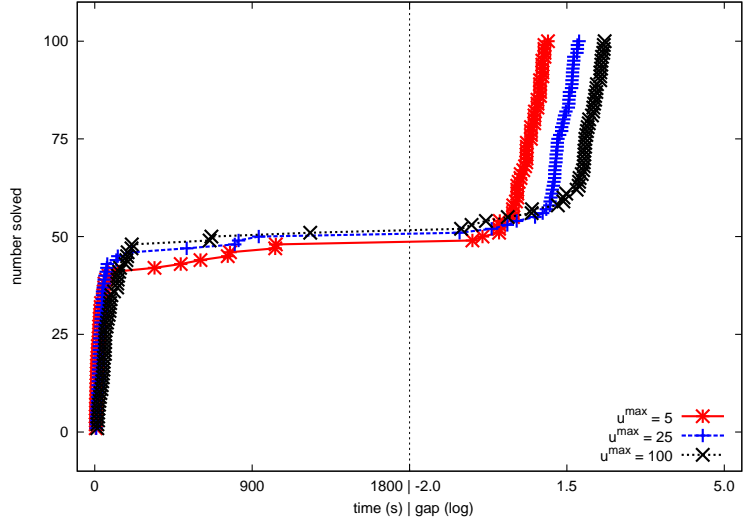}
	\caption{Cumulative distribution plot of performance, comparing \textbf{BCP} for varying $\utility^{\mathrm{max}}$ with $m = 6$ and $s = 4$.}
	\label{fig:PPMaxUtil}
\end{figure}

The figure exhibits an interesting relationship.  Parameter $\utility^{\mathrm{max}}$ sets the number of possible values that an agent can specify for the other individuals in the market.  Given more preference resolution/detail, it is slightly easier to find the optimal solutions. Additionally, the absolute gap increases for those instances that are unsolved.  This can be attributed to either the model becoming larger (as the size is directly related to $\utility^{\mathrm{max}}$) or that the relative differences from solution to solution is smaller when $\utility^{\mathrm{max}}$ is small.  The relative performance differences, however, are marginal, showing that even though model~(\ref{optmod:singleLevel-EXP}) is pseudo-polynomial in $\utility^{\mathrm{max}}$, \textbf{BCP} is able to scale for large preference ranges.

\subsection{Boundary effects on $\alpha$}
\label{sec:boundary}

When considering the single-objective variants of the TFP, when only total intra-group utility or stability, respectively, are of interest, a slight emphasis on the other objective leads to significant improvements. For $\alpha = 1$, maximizing utility only, a perturbation to $\alpha = 0.99$ results in slightly longer solution times but provides solutions with the same total utility but slightly better stability.   When $\alpha = 0$, and one is concerned only with stability, perturbing to $\alpha = 0.01$ results in some solutions with even \emph{better} stability within time limits, due to computational efficiency gains. 
 
Focusing first on the $\alpha = 1$ case, we restrict attention to the 197 instances (out of 220) solved to optimality by \textbf{BCP} with $\alpha = 1$ and $\alpha = 0.99$.  For every such instance, the total utility at the optimal solution is the same, but the maximum uplift can be significantly different. In 26 of the instances (13.2\%) a better solution with respect to maximum uplift is identified, with reduction in maximum uplift ranging from 3.84\% to 100\%, and an average of 35.2\%. For three instances in particular, the reported optimal solution left a maximum uplift of 8, 11, and 13, respectively under $\alpha = 1$, but with $\alpha = 0.99$, the reported optimal solution has the same total utility, but found an entirely stable solution (i.e., maximum uplift = 0, a reduction of 100\%).  This does come at a slight computational cost; the average solution time on these instances is 118 ($\alpha = 1$) and 147 ($\alpha = 0.99$) seconds, respectively.  This added computational effort, however, is rewarded with solutions of strictly better quality. This suggests, for example, that recent research using BCPP (i.e., $\alpha = 1$) in sports management \citep[e.g.,][]{Recalde2016} may find Pareto improved solutions using $\alpha = 0.99$.

The difference in algorithmic performance is even more apparent when comparing the solution quality and solution time under $\alpha = 0$ versus $\alpha = 0.01$.  First, only 73 instances are solved to optimality with $\alpha = 0$ versus 91 with $\alpha = 0.01$. More critically, despite having a theoretically worse optimal solution with respect to maximum uplift, in 131 of the 220 instances, \textbf{BCP} identifies a solution with smaller uplift when $\alpha = 0.01$ versus setting $\alpha = 0$, where the opposite occurs in only 3 instances. Furthermore, this reduction is often substantial.  The average percent reduction in maximum uplift in the 131 instances where \textbf{BCP} found a solution with smaller maximum uplift is 74.5\%.   In 59 of the 131 instances, the reduction is 100\%, meaning that a completely stable solution is found, whereas there are groups with positive uplift in the solution found with $\alpha = 0$ at 1800 seconds.

The reason for this surprising relative gain in solution quality can be attributed to the fact that solutions with high total intra-group utility will have relatively low instability, compared with randomly chosen groups.  Without any consideration of total team utility in the objective function, it is challenging for the solver to distinguish between partitions of agents.  Having a slight emphasis on total intra-group utility focuses the search for solutions, which turns out to be particularly effective under a column generation approach. 
\textbf{BCP}'s performance is only slightly hindered by the inclusion of characteristic constraints; \textbf{BCP} is able to solve 668 and 637 instances with and without characteristic constraints, respectively.  Of the 617 instances that mutually solved with and without the constraints, the average solution time is 141.3 and 142.2 seconds, respectively, an insignificant difference. 

\subsection{The effect of characteristic constraints}
\label{sec:CharCon}

\begin{figure}[t!]
	\centering
	\includegraphics[scale=0.8]{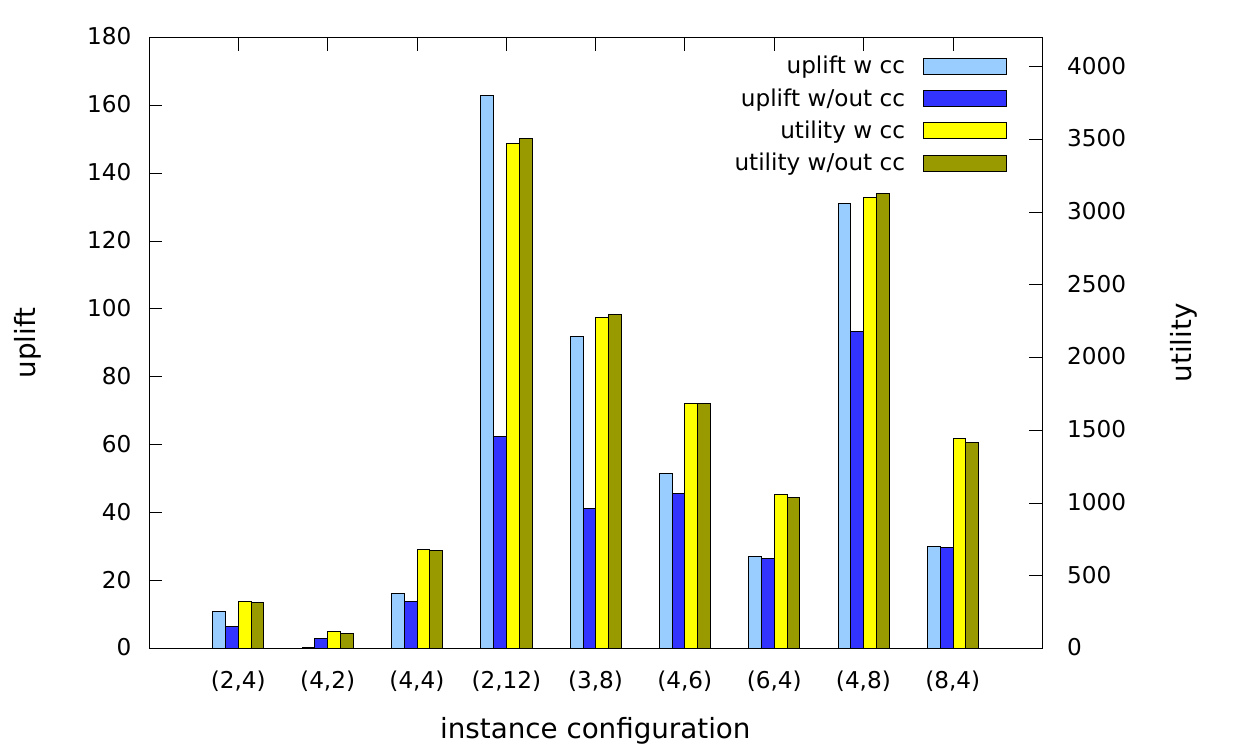}
	\caption{Average maximum uplift and average total market utility, with and without characteristic constraints.}
	\label{fig:CharPlotBar}
\end{figure}

Figure~\ref{fig:CharPlotBar} provides a visual summary depicting the average of the maximum uplift and total utility, respectively, in the best solutions obtained by \textbf{BCP} with and without characteristic constraints. Clearly uplift is much more sensitive to adding characteristic constraints than utility. Indeed, the change in average utility is barely noticeable when constraints are added in any case, compared with the maximum uplift, which rises drastically in several market sizes, even more than doubling for some configurations. (Table~\ref{tab:CharCons} in the Appendix provides further detail.) Market designers should take heed; the addition of characteristic constraints may seem innocuous if only total utility is considered, but actually it has the potential to introduce a great deal of instability. 


%

\section{Comparison of \textbf{BCP} to \textbf{DRAFT} and \textbf{OPOP} }
\label{sec:BenchComp}

Here we compare \textbf{BCP} (with different $\alpha$) to \textbf{DRAFT} and \textbf{OPOP} on three performance metrics: \emph{efficiency} (measured by average individual utility), \emph{inequity} (measured by the range of individual and team utility), and \emph{instability} (measured by individual and coalitional uplift). \textbf{DRAFT} and \textbf{OPOP} (see \S \ref{sec:litbench}) were not designed to handle characteristic constraints and may result in infeasiblity if applied na\"ively. (A team may for example over-consume a constrained characteristic, resulting in infeasbility for another team, even if they conscientiously pick to maintain their own feasibility throughout. Stated differently, given the NP-hardness of the problem, locally greedy algorithms will fail to maintain global feasibility.) The results of this section are therefore limited to an investigation of the $|\charSet| = 0$ subset of instances. 

%

\subsection{Efficiency}
\label{subsec:Eff}
A clustered bar chart provides efficiency comparisons of \textbf{BCP}, \textbf{DRAFT}, and \textbf{OPOP} in Figure~\ref{fig:efficiency}.
For each agent in each outcome, we divide agent $i$'s resulting utility by $i$'s maximum realizable utility $\individualMaxUtility{i}$. For each algorithm, Figure~\ref{fig:efficiency} reports the average of this percentage utility over all solutions for each generation scheme. 


Our results show that \textbf{BCP} solutions have superior average utility across all sets of instances. This superiority is more obvious for data generation schemes with independent (\textbf{G3}) and affinity preferences (\textbf{G4}) and larger $\objweight$. (The latter is to be expected as the weight of efficiency in the objective increases.) The significant common-value components of \textbf{G1} and \textbf{G2} result in a larger degree of ``necessary disappointment'' in which it is impossible for every agent to achieve her first best possible team. \textbf{BCP} tends to get closer to maximum satisfaction with the outcome of the mechanism. Perhaps surprisingly, even when the focus on efficiency is low (i.e., when $\alpha= 0.01$) \textbf{BCP} still provides more efficient solutions than the heuristic formats \textbf{DRAFT} and \textbf{OPOP}.

\begin{figure}[t!]
	\centering
	\includegraphics[scale=0.8]{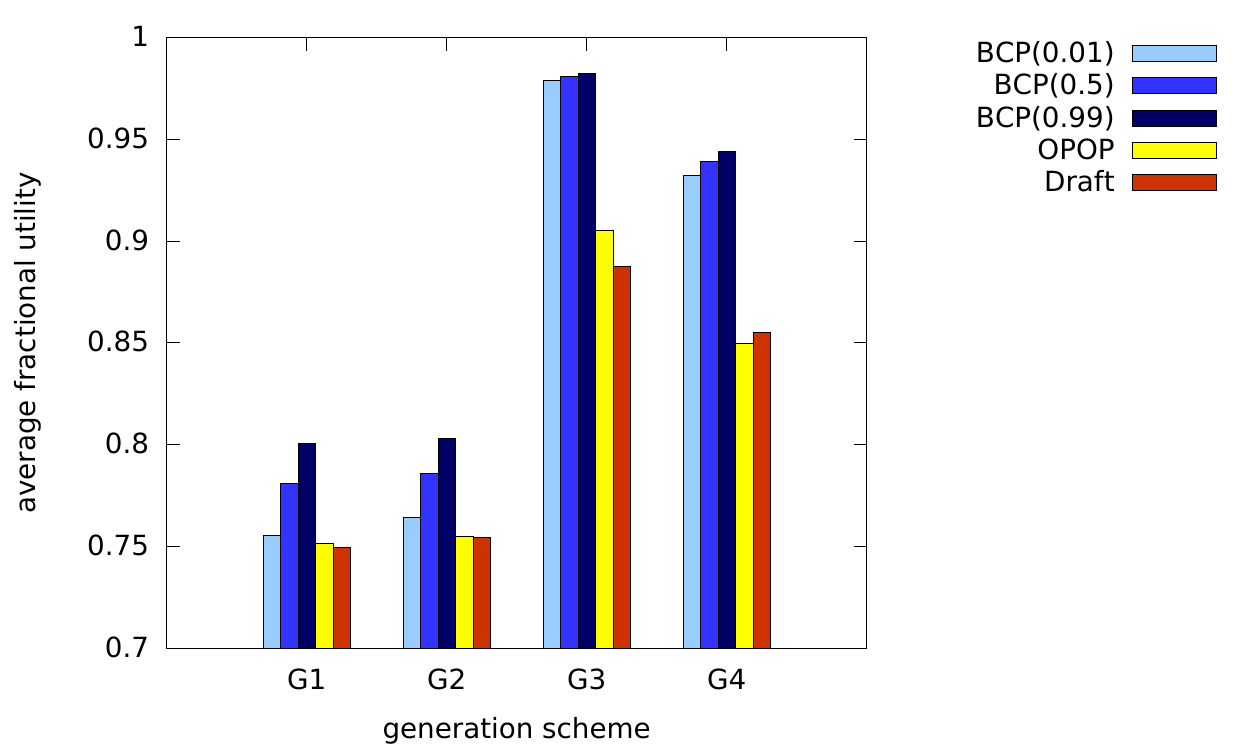}
	\caption{Average solution utility as a fraction of maximum realizable utility.}
	\label{fig:efficiency}
\end{figure}

\subsection{Inequity}
\label{fair}
Intuitively, having some agents with very high utility while other agents have very low utility can be problematic and viewed as inequitable or unfair. This can be measured and compared on the individual level and on the team level. For percentage-based comparisons, we define the \emph{individual inequity} as the best individual utility minus the worst individual utility in a solution, divided by the average of $\individualMaxUtility{i}$. Similarly, we define the \emph{team inequity} as the best team utility minus the worst team utility in a solution divided by the average \emph{ideal-team utility}, defined for each $i$ as the maximum utility of any coalition including $i$.  In Figure~\ref{fig:fairness} we report the average individual inequity and team inequity over all instances in each generation scheme for each of the five algorithms.  A taller bar indicates a worse solution on the measure of inequity. 


Clearly, the preference generation format can greatly impact the ranking of mechanisms on this measure of allocative inequity. The left chart in Figure~\ref{fig:fairness} indicates that for data generation schemes \textbf{G3} and \textbf{G4}, the individual inequity is typically smaller for \textbf{BCP} (regardless of $\objweight$), and thus its results may be perceived as fairer. However, for data generation schemes with stronger common-value effects (\textbf{G1} and \textbf{G2}), only \textbf{BCP} with $\objweight=0.99$ results in more equitable solutions than those found by \textbf{OPOP} and \textbf{DRAFT}.
\begin{figure}[t!] 
	\begin{minipage}[b]{0.5\linewidth}
		\centering
		\includegraphics[width=1\linewidth]{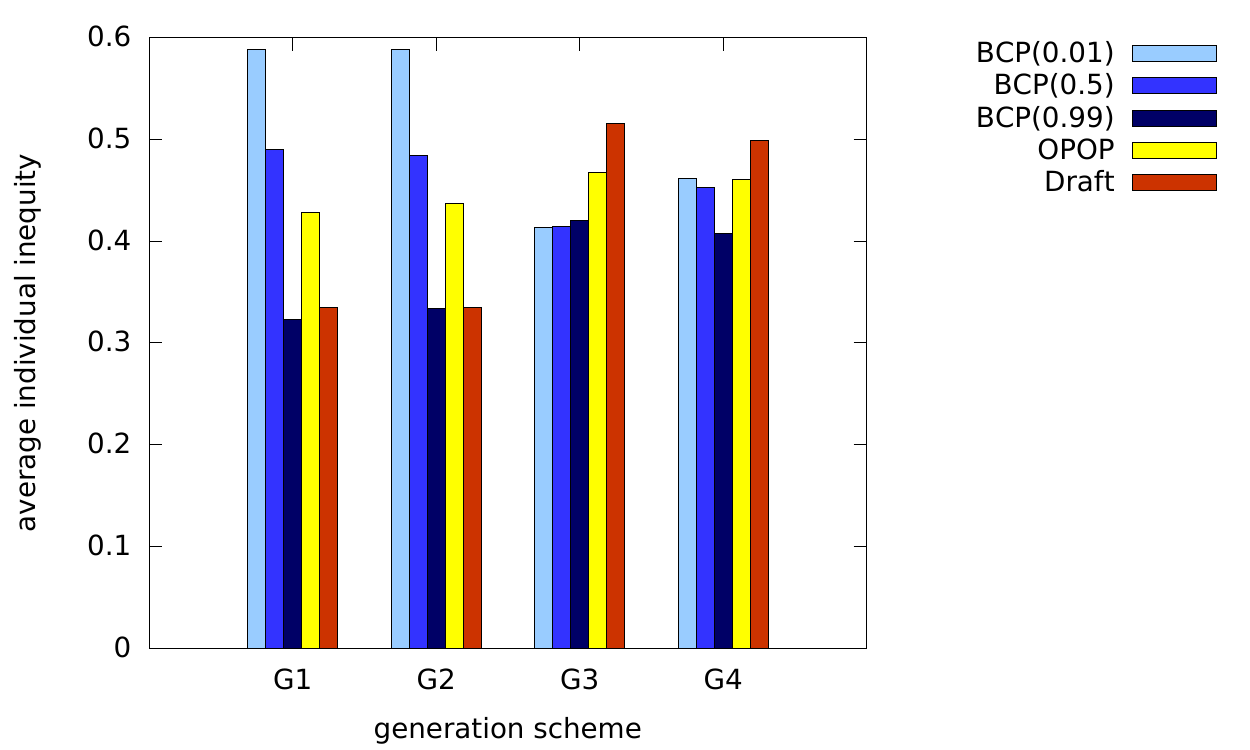} 
	\end{minipage}
	\begin{minipage}[b]{0.5\linewidth}
		\centering
		\includegraphics[width=1\linewidth]{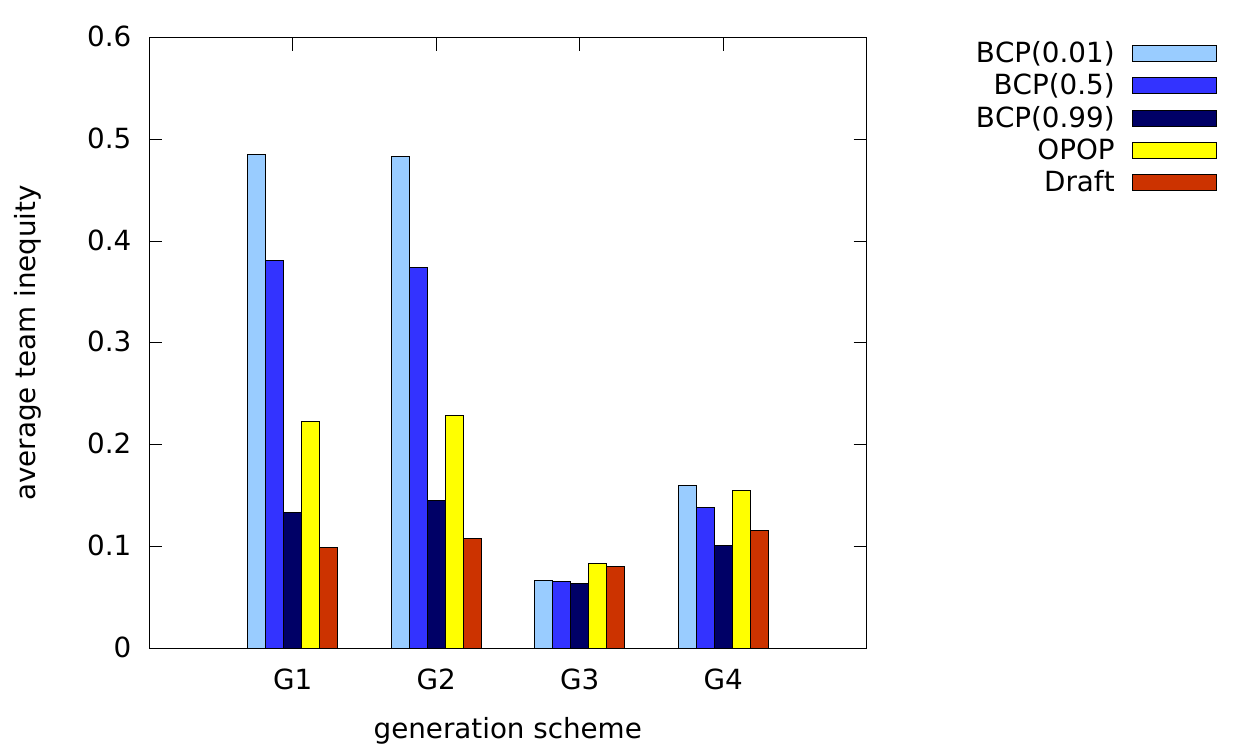} 
	\end{minipage}
	\caption{Average inequity (percentage utility range), averaged over individuals (left) and teams (right).} 
	\label{fig:fairness} 	
\end{figure}

Indeed this relationship is among our more interesting findings; the focus on stability (inherent as $\objweight$ decreases) results in less equitable solutions when there is wide agreement on the ``top of the market'' as in \textbf{G1} and \textbf{G2}. Stated differently, in order to reduce the tendency of groups to want to break away and form their own teams (when the focus is stability), one will have to generate some teams with more ``top'' agents together, resulting in a wider satisfaction gap (inequity) in the market. This can only occur when the underlying preferences reflect a certain degree of \emph{agreement} on which agents comprise the top of the market (e.g., in \textbf{G1} and \textbf{G2}). When agents tend to prefer to gather together based on their own affinity for those of similar characteristics (\textbf{G4}), this effect is reduced, and it nearly disappears when preferences are independent (\textbf{G3}). 

This effect is even more pronounced in the right half of Figure~\ref{fig:fairness}. Our results show that for data generation schemes with independent and affinity preferences (\textbf{G3} and \textbf{G4}), \textbf{BCP} with $\objweight=0.99$ achieves more equitable teams than other algorithms. However, for data generation schemes \textbf{G1} and \textbf{G2}, \textbf{DRAFT} achieves the the most equitable solutions across teams. Note that \textbf{BCP} is not directly minimizing inequity, and in particular when we focus only on minimizing the maximum uplift ($\objweight = 0.01$), the team inequity is more obvious. \textbf{DRAFT} tends toward teams of equal satisfaction, spreading out ``top picks'' as would be expected; \textbf{BCP} with $\objweight = 0.01$ tends to allow ``top picks'' to clump together in order to reduce their desire to leave the market.

\subsection{Instability}

To compare team formation outcomes across instances we again need a metric scaled to 1. Let \emph{individual instability} be defined as the maximum uplift an agent can gain by defecting to a hypothetical uplift team, divided by her utility within her current team. \emph{Team instability} is the maximum uplift of a team containing each $i$ divided by the utility of her current team.  These provide a relative measure of how much an individual or a team is incentivized to defect given the current team assignment. Figure~\ref{fig:stab1} provides two plots depicting the average individual (left) and average coalition uplift (right), averaged over all $i$.  
 
 \begin{figure}[t!] 
	\begin{minipage}[b]{0.5\linewidth}
		\centering
		\includegraphics[width=1\linewidth]{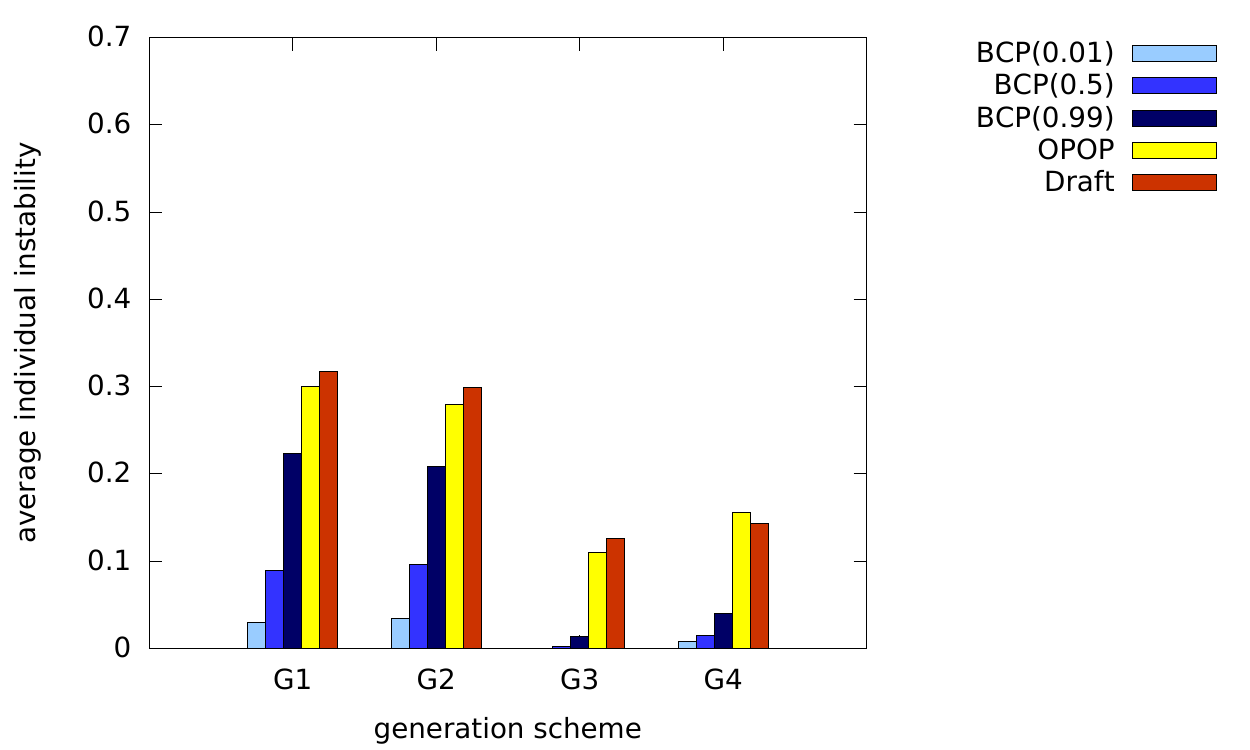}
	\end{minipage}
	\begin{minipage}[b]{0.5\linewidth}
		\centering
		\includegraphics[width=1\linewidth]{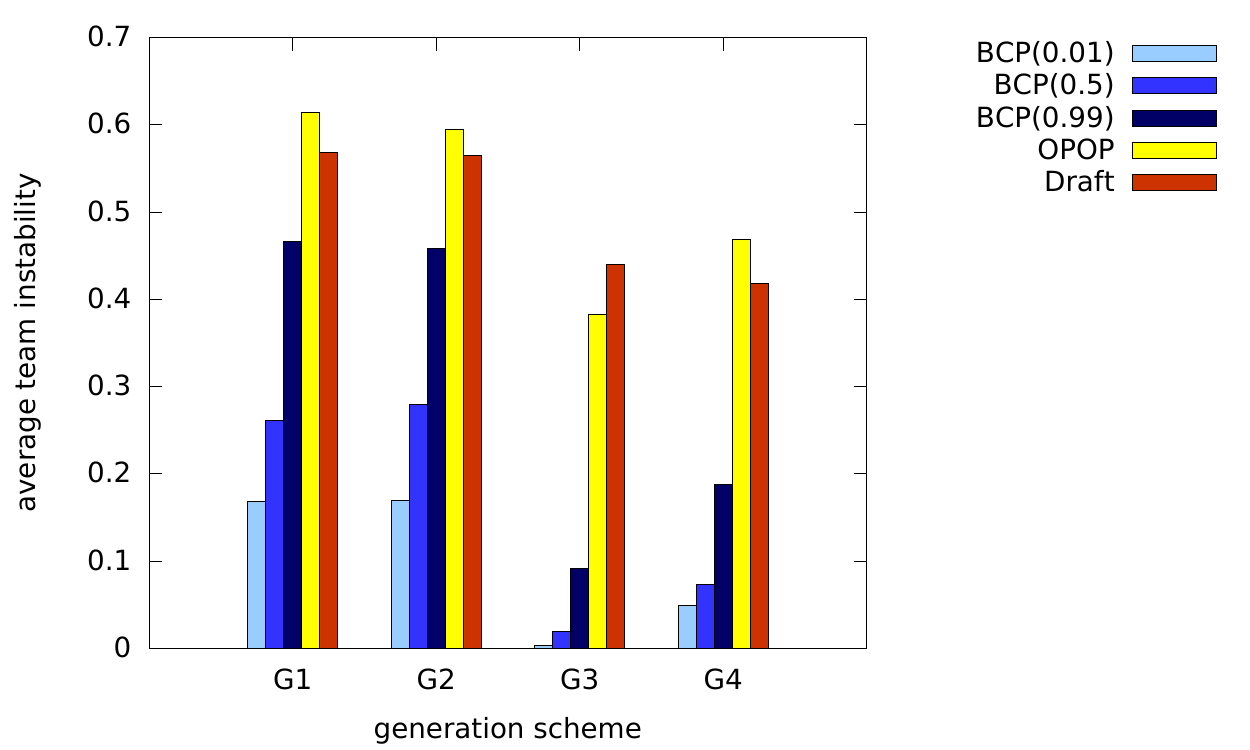} 
	\end{minipage}	
		\caption{Average individual instability (left) and average team instability (right)} 
		\label{fig:stab1}
\end{figure}




These charts show that regardless of $\objweight$, the solutions obtained by \textbf{BCP} are more stable, having lower average and maximum individual uplifts. Moreover, the average individual benefit of forming a coalition in the solutions obtained by \textbf{BCP} with $\alpha=0.01$ is less than 4\% of their current benefit, while in \textbf{OPOP} and \textbf{DRAFT}, it can be on average as high as 30\% of their current benefit.  Also, note that the maximum uplift is significantly lower in data generation schemes \textbf{G3} and \textbf{G4} in comparison to data \textbf{G1} and \textbf{G2}. (Though not shown, the relative rankings do not change if we focus on \emph{maximum} uplift as oppopsed to \emph{average}.)


%
%


\section{Conclusion and Future Work}
\label{sec:conclusion}
In this paper we investigate computational models for the team formation problem.  In order to balance intra-team utility with solution stability, we formulated a bi-level binary optimization model and developed a branch-cut-and-price solution algorithm. The pseudo-polynomial approach to the bi-level problem is itself an interesting contribution, and we detailed how to implement the resulting algorithm, which still has to manage an exponential number of variables \emph{and} constraints, demanding advanced computational methods, which we outline. Experimental results indicated that the proposed algorithm \textbf{BCP} is particularly effective at finding high-quality solutions quickly. Stability as an objective to be optimized over remains a particularly challenging computational problem, but we have shown a promising new approach.

Our results also indicate that ignoring stability can result in inferior solutions when Pareto improvements are available. Because a defection-based measure of instability in this context points to opportunities for profitable group deviation (i.e., strategic manipulation), these results will have practical implications. We showed that heuristic algorithms from the literature may be leaving efficiency and stability gains on the table, as might have been expected. Yet heuristics based on draft principles may be common in practice; we would argue that this has much to do with the inherent computational difficulty of optimization approaches, which we have shown can be mitigated by algorithms like \textbf{BCP}. We have shown that our new methods can improve total satisfaction, and improve incentives via reduced instability measures.

While \textbf{BCP} with $\objweight = 0.99$ performed very well on equity as a measure of fairness, we found an interesting trade-off, in which the market maker must in some cases accept inter-agent inequity in pursuit of stability (as $\objweight$ goes to zero), in particular it would seem, where common value is the primary driver of preferences. This discrepancy of inequity performance across preference models proved interesting, and we expect future studies to shed more light on how different models or descriptions of preferences affect the performance of various algorithms for team formation.

\clearpage

\bibliographystyle{plainnat} 
\bibliography{tfpbib} 

\clearpage 

%

\section{Appendix}

Table~\ref{tab:G1Details} through Table~\ref{tab:G4Details} provide detailed solution statistics on \textbf{G1}, \textbf{G2}, \textbf{G3}, and \textbf{G4}, respectively.  The first four columns report the number of agents, the number of groups, the maximum pairwise utility, and $\alpha$.  The next three sets of four columns report solution statistics for \textbf{EXP}, \textbf{BC}, and \textbf{BCP}, respectively. In particular, for each algorithm and each configuration, we report the number of instances solved in 1800 seconds ($n^s$), the average solution time over those $n^2$ instances ($\overline{\textnormal{time}}$), the number of instances that have a provable optimality gap $n^g$ (i.e., for \textbf{EXP} the number of instances that don't hit memory limits and for \textbf{BCP} the number of instances for which the root node was solved and an initial upper bound is proven), and finally the average gap over those $n^g$ instances.  The last two columns report the average total utility ($\utility^{\mathrm{best}}$) and maximum uplift ($\uplift^{\mathrm{best}}$) for the best solution identified by \textbf{BCP} for that configuration. 

Table~\ref{tab:CharCons} provides details on the effect of including characteristic constraints
(considering only those instances with $\utility^{\textnormal{max}} = 25$).  For each instance configuration and for $\alpha \in \set{0.01, 0.5, 0.99}$, the table reports, for the best solution identified by \textbf{BCP}, the average uplift ($\bar{r}$) and average total utility ($\bar{u}$) with and without characteristic constraints, for the instances from each generation scheme, in sequence.


\clearpage

\begin{table}[t!]
	\centering
	\tiny
	\caption{Detailed solution statistics for \textbf{G1}}
	\label{tab:G1Details}
	\begin{tabular}{llll|llll|llll|llll|ll}
		&&&&
		\multicolumn{4}{c|}{\textbf{EXP}} &
		\multicolumn{4}{c|}{\textbf{BC}} &
		\multicolumn{4}{c|}{\textbf{BCP}} &
		\multicolumn{2}{c}{ }
		\\ 
		$n$ &
		$m$ &
		$\utility^{\mathrm{max}}$ &
		$\alpha$ &
		$n^s$ &
		$\overline{\textnormal{time}}$ &
		$n^g$
		&
		$\overline{\textnormal{gap}}$ &
		$n^s$ &
		$\overline{\textnormal{time}}$ &
		$n^g$
		&
		$\overline{\textnormal{gap}}$ &
		$n^s$ &
		$\overline{\textnormal{time}}$ &
		$n^g$
		&
		$\overline{\textnormal{gap}}$ &
		$\utility^{\mathrm{best}}$ &
		$\uplift^{\mathrm{best}}$
		\\ \hline
		8  & 2 & 25  & 0    & 5 & 0.07  & 5 & 0     & 5 & 0.01 & 5 & 0     & 5 & 0.09  & 5 & 0     & 302   & 0     \\
		&   &     & 0.01 & 5 & 0.45  & 5 & 0     & 5 & 0.92 & 5 & 0     & 5 & 2.61  & 5 & 0     & 304.2 & 0     \\
		&   &     & 0.5  & 5 & 0.49  & 5 & 0     & 5 & 0.44 & 5 & 0     & 5 & 2.74  & 5 & 0     & 304.2 & 0     \\
		&   &     & 0.99 & 5 & 0.09  & 5 & 0     & 5 & 0.92 & 5 & 0     & 5 & 0.21  & 5 & 0     & 319   & 36.6  \\
		&   &     & 1    & 5 & 0.04  & 5 & 0     & 5 & 0.43 & 5 & 0     & 5 & 0.19  & 5 & 0     & 319   & 36.6  \\ \hline
		& 4 & 25  & 0    & 5 & 0.01  & 5 & 0     & 5 & 0.74 & 5 & 0     & 5 & 0.02  & 5 & 0     & 109   & 0     \\
		&   &     & 0.01 & 5 & 0.01  & 5 & 0     & 5 & 0.73 & 5 & 0     & 5 & 0.13  & 5 & 0     & 109   & 0     \\
		&   &     & 0.5  & 5 & 0.01  & 5 & 0     & 5 & 1.1  & 5 & 0     & 5 & 0.12  & 5 & 0     & 110.6 & 1     \\
		&   &     & 0.99 & 5 & 0.01  & 5 & 0     & 5 & 1.37 & 5 & 0     & 5 & 0.03  & 5 & 0     & 113.4 & 9.8   \\
		&   &     & 1    & 5 & 0.01  & 5 & 0     & 5 & 1.24 & 5 & 0     & 5 & 0.03  & 5 & 0     & 113.4 & 9.8   \\ \hline
		16 & 4 & 25  & 0    & 5 & 481.8 & 5 & 0     & 1 & 1607 & 5 & 23.2  & 0 & -     & 5 & 20.4  & 612.8 & 20.4  \\
		&   &     & 0.01 & 5 & 391.9 & 5 & 0     & 0 & -    & 5 & 16.48 & 1 & 832   & 5 & 3.51  & 633.4 & 3.4   \\
		&   &     & 0.5  & 5 & 528.2 & 5 & 0     & 0 & -    & 5 & 216.1 & 1 & 831.7 & 5 & 4.39  & 651.8 & 12.8  \\
		&   &     & 0.99 & 5 & 4.05  & 5 & 0     & 0 & -    & 5 & 371   & 5 & 4.42  & 5 & 0     & 666.6 & 47    \\
		&   &     & 1    & 5 & 0.62  & 5 & 0     & 0 & -    & 5 & 346.8 & 5 & 3.83  & 5 & 0     & 666.6 & 47.6  \\ \hline
		24 & 2 & 25  & 0    & 0 & -     & 0 & -     & 5 & 0.22 & 5 & 0     & 5 & 1.14  & 5 & 0     & 3313  & 0     \\
		&   &     & 0.01 & 0 & -     & 0 & -     & 0 & -    & 5 & 4.2   & 0 & -     & 4 & 1.01  & 3313  & 0     \\
		&   &     & 0.5  & 0 & -     & 0 & -     & 0 & -    & 5 & 225.3 & 0 & -     & 3 & 47.61 & 3313  & 0     \\
		&   &     & 0.99 & 0 & -     & 0 & -     & 0 & -    & 5 & 304.7 & 3 & 1346  & 3 & 0     & 3429  & 370.4 \\
		&   &     & 1    & 0 & -     & 0 & -     & 0 & -    & 5 & 367.2 & 0 & -     & 0 & -     & 3429  & 370.4 \\ \hline
		& 3 & 25  & 0    & 0 & -     & 0 & -     & 5 & 0.29 & 5 & 0     & 5 & 0.69  & 5 & 0     & 2094  & 0     \\ 
		&   &     & 0.01 & 0 & -     & 0 & -     & 0 & -    & 5 & 17.63 & 0 & -     & 5 & 1.48  & 2094  & 0     \\
		&   &     & 0.5  & 0 & -     & 0 & -     & 0 & -    & 5 & 864.1 & 0 & -     & 5 & 73.76 & 2094  & 0     \\
		&   &     & 0.99 & 0 & -     & 0 & -     & 0 & -    & 5 & 1605  & 5 & 1124  & 5 & 0     & 2246  & 209.2 \\
		&   &     & 1    & 0 & -     & 0 & -     & 0 & -    & 5 & 1650  & 5 & 633.7 & 5 & 0     & 2246  & 209.2 \\ \hline
		& 4 & 25  & 0    & 0 & -     & 0 & -     & 0 & -    & 5 & 124.6 & 0 & -     & 5 & 49.6  & 1487  & 49.6  \\
		&   &     & 0.01 & 0 & -     & 0 & -     & 0 & -    & 5 & 114.1 & 0 & -     & 5 & 45.56 & 1510  & 44.8  \\
		&   &     & 0.5  & 0 & -     & 0 & -     & 0 & -    & 5 & 709.7 & 0 & -     & 5 & 40.6  & 1625  & 75.8  \\
		&   &     & 0.99 & 0 & -     & 0 & -     & 0 & -    & 5 & 1283  & 5 & 228.8 & 5 & 0     & 1637  & 126.6 \\
		&   &     & 1    & 5 & 130   & 5 & 0     & 0 & -    & 5 & 1289  & 5 & 184.1 & 5 & 0     & 1637  & 126.6 \\ \hline
		& 6 & 5   & 0    & 0 & -     & 5 & 4.8   & 0 & -    & 5 & 8.2   & 0 & -     & 5 & 8.8   & 179.4 & 8.8   \\
		&   &     & 0.01 & 0 & -     & 5 & 6.74  & 0 & -    & 5 & 11.06 & 0 & -     & 5 & 4.43  & 200.6 & 4.4   \\
		&   &     & 0.5  & 1 & 1122  & 5 & 1.2   & 0 & -    & 5 & 86.3  & 0 & -     & 5 & 2.77  & 209.6 & 7.2   \\
		&   &     & 0.99 & 5 & 27.35 & 5 & 0     & 0 & -    & 5 & 160.7 & 5 & 50.38 & 5 & 0     & 210.6 & 9     \\
		&   &     & 1    & 5 & 1.67  & 5 & 0     & 0 & -    & 5 & 160.2 & 5 & 17.98 & 5 & 0     & 210.6 & 10.4  \\ \hline
		&   & 25  & 0    & 0 & -     & 5 & 42.2  & 0 & -    & 5 & 63.8  & 0 & -     & 5 & 33.2  & 903.6 & 33.2  \\
		&   &     & 0.01 & 0 & -     & 5 & 31.32 & 0 & -    & 5 & 49.34 & 0 & -     & 5 & 20.35 & 999   & 20.4  \\
		&   &     & 0.5  & 0 & -     & 5 & 20.9  & 0 & -    & 5 & 363.9 & 0 & -     & 5 & 16.3  & 1010  & 28.4  \\
		&   &     & 0.99 & 5 & 50.38 & 5 & 0     & 0 & -    & 5 & 676.5 & 5 & 35.75 & 5 & 0     & 1022  & 56.2  \\
		&   &     & 1    & 5 & 4.99  & 5 & 0     & 0 & -    & 5 & 681.8 & 5 & 31.1  & 5 & 0     & 1022  & 56.2  \\ \hline
		&   & 100 & 0    & 0 & -     & 5 & 250.4 & 0 & -    & 5 & 245   & 0 & -     & 5 & 133.8 & 3620  & 133.8 \\
		&   &     & 0.01 & 0 & -     & 5 & 131.1 & 0 & -    & 5 & 217.9 & 0 & -     & 5 & 69.01 & 3862  & 67.4  \\
		&   &     & 0.5  & 0 & -     & 5 & 119.1 & 0 & -    & 5 & 1565  & 0 & -     & 5 & 74.54 & 4063  & 129.2 \\
		&   &     & 0.99 & 5 & 508.7 & 5 & 0     & 0 & -    & 5 & 2881  & 5 & 107.4 & 5 & 0     & 4096  & 229.6 \\
		&   &     & 1    & 5 & 14.98 & 5 & 0     & 0 & -    & 5 & 2899  & 5 & 82.16 & 5 & 0     & 4096  & 229.6 \\ \hline
		32 & 4 & 25  & 0    & 0 & -     & 0 & -     & 1 & 0.85 & 5 & 207.2 & 1 & 1.29  & 5 & 65    & 2777  & 65    \\
		&   &     & 0.01 & 0 & -     & 0 & -     & 0 & -    & 5 & 172.1 & 0 & -     & 0 & -     & 2769  & 66.4  \\
		&   &     & 0.5  & 0 & -     & 0 & -     & 0 & -    & 5 & 1325  & 0 & -     & 0 & -     & 2984  & 171.4 \\
		&   &     & 0.99 & 0 & -     & 0 & -     & 0 & -    & 5 & 2413  & 0 & -     & 0 & -     & 3028  & 240.8 \\
		&   &     & 1    & 0 & -     & 0 & -     & 0 & -    & 5 & 2442  & 0 & -     & 0 & -     & 3028  & 233.4 \\ \hline
		& 8 & 25  & 0    & 0 & -     & 5 & 63.6  & 0 & -    & 5 & 65    & 0 & -     & 5 & 41.6  & 1199  & 41.6  \\
		&   &     & 0.01 & 0 & -     & 5 & 39.59 & 0 & -    & 5 & 69.91 & 0 & -     & 5 & 27.58 & 1298  & 27    \\
		&   &     & 0.5  & 0 & -     & 5 & 30.3  & 0 & -    & 5 & 550.2 & 0 & -     & 5 & 25.56 & 1378  & 45.8  \\
		&   &     & 0.99 & 5 & 465.7 & 5 & 0     & 0 & -    & 5 & 1024  & 5 & 164.7 & 5 & 0     & 1385  & 64.4  \\
		&   &     & 1    & 5 & 15.7  & 5 & 0     & 0 & -    & 5 & 1036  & 5 & 109.5 & 5 & 0     & 1385  & 65.2 
	\end{tabular}
\end{table}

\begin{table}[]
	\centering
	\tiny
	\caption{Detailed solution statistics for \textbf{G2}}
	\label{tab:G2Details}
	\begin{tabular}{llll|llll|llll|llll|ll}
		&&&&
		\multicolumn{4}{c|}{\textbf{EXP}} &
		\multicolumn{4}{c|}{\textbf{BC}} &
		\multicolumn{4}{c|}{\textbf{BCP}} &
		\multicolumn{2}{c}{ }
		\\ 
		$n$ &
		$m$ &
		$\utility^{\mathrm{max}}$ &
		$\alpha$ &
		$n^s$ &
		$\overline{\textnormal{time}}$ &
		$n^g$
		&
		$\overline{\textnormal{gap}}$ &
		$n^s$ &
		$\overline{\textnormal{time}}$ &
		$n^g$
		&
		$\overline{\textnormal{gap}}$ &
		$n^s$ &
		$\overline{\textnormal{time}}$ &
		$n^g$
		&
		$\overline{\textnormal{gap}}$ &
		$\utility^{\mathrm{best}}$ &
		$\uplift^{\mathrm{best}}$
		\\ \hline
		8  & 2 & 25  & 0    & 5 & 0.07  & 5 & 0     & 5 & 0.12 & 5 & 0     & 5 & 0.49  & 5 & 0     & 300.8 & 0     \\
		&   &     & 0.01 & 5 & 0.23  & 5 & 0     & 5 & 0.52 & 5 & 0     & 5 & 3.41  & 5 & 0     & 305   & 0     \\
		&   &     & 0.5  & 5 & 0.36  & 5 & 0     & 5 & 0.42 & 5 & 0     & 5 & 3.05  & 5 & 0     & 305   & 0     \\
		&   &     & 0.99 & 5 & 0.09  & 5 & 0     & 5 & 0.62 & 5 & 0     & 5 & 0.17  & 5 & 0     & 315.8 & 29.6  \\
		&   &     & 1    & 5 & 0.04  & 5 & 0     & 5 & 0.45 & 5 & 0     & 5 & 0.16  & 5 & 0     & 315.8 & 29.6  \\ \hline
		& 4 & 25  & 0    & 5 & 0.01  & 5 & 0     & 5 & 0.72 & 5 & 0     & 5 & 1.02  & 5 & 0     & 106.6 & 2     \\
		&   &     & 0.01 & 5 & 0.01  & 5 & 0     & 5 & 0.72 & 5 & 0     & 5 & 0.97  & 5 & 0     & 106.6 & 2     \\
		&   &     & 0.5  & 5 & 0.01  & 5 & 0     & 5 & 0.89 & 5 & 0     & 5 & 0.3   & 5 & 0     & 110.4 & 2.4   \\
		&   &     & 0.99 & 5 & 0.01  & 5 & 0     & 5 & 0.96 & 5 & 0     & 5 & 0.04  & 5 & 0     & 113.2 & 8.6   \\
		&   &     & 1    & 5 & 0.01  & 5 & 0     & 5 & 0.95 & 5 & 0     & 5 & 0.03  & 5 & 0     & 113.2 & 9.2   \\ \hline
		16 & 4 & 25  & 0    & 5 & 698.3 & 5 & 0     & 1 & 0.08 & 5 & 26.8  & 1 & 0.22  & 5 & 16.4  & 612.4 & 16.4  \\
		&   &     & 0.01 & 4 & 433.5 & 5 & 3.8   & 0 & -    & 5 & 17.43 & 1 & 449.1 & 5 & 6.67  & 636.2 & 6.6   \\
		&   &     & 0.5  & 5 & 520.3 & 5 & 0     & 0 & -    & 5 & 211.9 & 1 & 528.7 & 5 & 5.13  & 652.8 & 14    \\
		&   &     & 0.99 & 5 & 3.1   & 5 & 0     & 0 & -    & 5 & 376.4 & 5 & 5.66  & 5 & 0     & 668   & 44.4  \\
		&   &     & 1    & 5 & 0.62  & 5 & 0     & 0 & -    & 5 & 364.6 & 5 & 4.47  & 5 & 0     & 668   & 44.4  \\ \hline
		24 & 2 & 25  & 0    & 0 & -     & 0 & -     & 5 & 0.22 & 5 & 0     & 5 & 1.18  & 5 & 0     & 3305  & 0     \\
		&   &     & 0.01 & 0 & -     & 0 & -     & 0 & -    & 5 & 4.43  & 0 & -     & 2 & 0.95  & 3305  & 0     \\
		&   &     & 0.5  & 0 & -     & 0 & -     & 0 & -    & 5 & 208.9 & 0 & -     & 2 & 46.25 & 3305  & 0     \\
		&   &     & 0.99 & 0 & -     & 0 & -     & 0 & -    & 5 & 367.3 & 1 & 1315  & 2 & 1.07  & 3419  & 369.6 \\
		&   &     & 1    & 0 & -     & 0 & -     & 0 & -    & 5 & 355   & 0 & -     & 0 & -     & 3419  & 369.6 \\ \hline
		& 3 & 25  & 0    & 0 & -     & 0 & -     & 4 & 0.27 & 5 & 48.4  & 4 & 0.68  & 5 & 9.8   & 2101  & 9.8   \\
		&   &     & 0.01 & 0 & -     & 0 & -     & 0 & -    & 5 & 29.77 & 0 & -     & 5 & 19.67 & 2096  & 18.4  \\
		&   &     & 0.5  & 0 & -     & 0 & -     & 0 & -    & 5 & 856.2 & 0 & -     & 5 & 71.42 & 2124  & 27.6  \\
		&   &     & 0.99 & 0 & -     & 0 & -     & 0 & -    & 5 & 1633  & 5 & 834.5 & 5 & 0     & 2247  & 202.6 \\
		&   &     & 1    & 0 & -     & 0 & -     & 0 & -    & 5 & 1644  & 5 & 827.7 & 5 & 0     & 2247  & 202.6 \\ \hline
		& 4 & 25  & 0    & 0 & -     & 0 & -     & 0 & -    & 5 & 128.2 & 0 & -     & 5 & 34.8  & 1502  & 34.8  \\
		&   &     & 0.01 & 0 & -     & 0 & -     & 0 & -    & 5 & 105.2 & 0 & -     & 4 & 35.59 & 1525  & 43.6  \\
		&   &     & 0.5  & 0 & -     & 0 & -     & 0 & -    & 5 & 668.3 & 0 & -     & 5 & 44.38 & 1622  & 74.8  \\
		&   &     & 0.99 & 0 & -     & 0 & -     & 0 & -    & 5 & 1200  & 5 & 301.8 & 5 & 0     & 1642  & 118   \\
		&   &     & 1    & 5 & 138.6 & 5 & 0     & 0 & -    & 5 & 1223  & 5 & 172.3 & 5 & 0     & 1642  & 118   \\ \hline
		& 6 & 5   & 0    & 0 & -     & 5 & 4     & 0 & -    & 5 & 7.8   & 0 & -     & 5 & 7.2   & 183.2 & 7.2   \\
		&   &     & 0.01 & 0 & -     & 5 & 5.5   & 0 & -    & 5 & 11.1  & 0 & -     & 5 & 5.41  & 201   & 5.4   \\
		&   &     & 0.5  & 2 & 1462  & 5 & 1.8   & 0 & -    & 5 & 71.1  & 0 & -     & 5 & 2.85  & 208.6 & 6.8   \\
		&   &     & 0.99 & 5 & 27.31 & 5 & 0     & 0 & -    & 5 & 129.4 & 5 & 39.32 & 5 & 0     & 210.2 & 9.8   \\
		&   &     & 1    & 5 & 1.48  & 5 & 0     & 0 & -    & 5 & 130.8 & 5 & 15.19 & 5 & 0     & 210.2 & 11.8  \\ \hline
		&   & 25  & 0    & 0 & -     & 5 & 37.6  & 0 & -    & 5 & 61.8  & 0 & -     & 5 & 39    & 904   & 39    \\
		&   &     & 0.01 & 0 & -     & 5 & 32.72 & 0 & -    & 5 & 54.94 & 0 & -     & 5 & 22.39 & 969.4 & 22.2  \\
		&   &     & 0.5  & 0 & -     & 5 & 21.9  & 0 & -    & 5 & 366.2 & 0 & -     & 5 & 18.33 & 1010  & 37.4  \\
		&   &     & 0.99 & 5 & 56.24 & 5 & 0     & 0 & -    & 5 & 673.6 & 5 & 45.67 & 5 & 0     & 1017  & 58.4  \\
		&   &     & 1    & 5 & 4.67  & 5 & 0     & 0 & -    & 5 & 684.2 & 5 & 27.76 & 5 & 0     & 1017  & 63.6  \\ \hline
		&   & 100 & 0    & 0 & -     & 5 & 229.8 & 0 & -    & 5 & 252   & 0 & -     & 5 & 133.6 & 3673  & 133.6 \\
		&   &     & 0.01 & 0 & -     & 5 & 134.6 & 0 & -    & 5 & 230.1 & 0 & -     & 5 & 83.74 & 3964  & 83.4  \\
		&   &     & 0.5  & 0 & -     & 5 & 104.3 & 0 & -    & 5 & 1616  & 0 & -     & 5 & 76.22 & 4074  & 144   \\
		&   &     & 0.99 & 5 & 150.7 & 5 & 0     & 0 & -    & 5 & 2960  & 5 & 90.41 & 5 & 0     & 4098  & 218.2 \\
		&   &     & 1    & 5 & 14.22 & 5 & 0     & 0 & -    & 5 & 2929  & 5 & 50.88 & 5 & 0     & 4098  & 218.2 \\ \hline
		32 & 4 & 25  & 0    & 0 & -     & 0 & -     & 3 & 0.85 & 5 & 92.8  & 3 & 1.39  & 5 & 21.6  & 2792  & 21.6  \\
		&   &     & 0.01 & 0 & -     & 0 & -     & 0 & -    & 5 & 107.2 & 0 & -     & 0 & -     & 2792  & 21.6  \\
		&   &     & 0.5  & 0 & -     & 0 & -     & 0 & -    & 5 & 1294  & 0 & -     & 0 & -     & 2963  & 172   \\
		&   &     & 0.99 & 0 & -     & 0 & -     & 0 & -    & 5 & 2418  & 0 & -     & 0 & -     & 3011  & 247.4 \\ 
		&   &     & 1    & 0 & -     & 0 & -     & 0 & -    & 5 & 2434  & 0 & -     & 0 & -     & 3019  & 236.4 \\ \hline
		& 8 & 25  & 0    & 0 & -     & 5 & 65.8  & 0 & -    & 5 & 59.2  & 0 & -     & 5 & 39.4  & 1211  & 39.4  \\
		&   &     & 0.01 & 0 & -     & 5 & 40.79 & 0 & -    & 5 & 75.04 & 0 & -     & 5 & 22.3  & 1317  & 21.8  \\
		&   &     & 0.5  & 0 & -     & 5 & 31.8  & 0 & -    & 5 & 527.6 & 0 & -     & 5 & 22    & 1376  & 34.8  \\
		&   &     & 0.99 & 5 & 594.2 & 5 & 0     & 0 & -    & 5 & 954.2 & 5 & 129.2 & 5 & 0     & 1389  & 64    \\
		&   &     & 1    & 5 & 15.15 & 5 & 0     & 0 & -    & 5 & 971.8 & 5 & 105.2 & 5 & 0     & 1389  & 64   
	\end{tabular}
\end{table}

\begin{table}[]
	\centering
	\tiny
	\caption{Detailed solution statistics for \textbf{G3}}
	\label{tab:G3Details}
	\begin{tabular}{llll|llll|llll|llll|ll}
		&&&&
		\multicolumn{4}{c|}{\textbf{EXP}} &
		\multicolumn{4}{c|}{\textbf{BC}} &
		\multicolumn{4}{c|}{\textbf{BCP}} &
		\multicolumn{2}{c}{ }
		\\ 
		$n$ &
		$m$ &
		$\utility^{\mathrm{max}}$ &
		$\alpha$ &
		$n^s$ &
		$\overline{\textnormal{time}}$ &
		$n^g$
		&
		$\overline{\textnormal{gap}}$ &
		$n^s$ &
		$\overline{\textnormal{time}}$ &
		$n^g$
		&
		$\overline{\textnormal{gap}}$ &
		$n^s$ &
		$\overline{\textnormal{time}}$ &
		$n^g$
		&
		$\overline{\textnormal{gap}}$ &
		$\utility^{\mathrm{best}}$ &
		$\uplift^{\mathrm{best}}$
		\\ \hline
		8  & 2 & 25  & 0    & 5 & 0.29  & 5 & 0     & 5 & 1.26  & 5 & 0     & 5 & 0.35  & 5 & 0     & 325.2 & 0     \\
		&   &     & 0.01 & 5 & 0.09  & 5 & 0     & 5 & 1.16  & 5 & 0     & 5 & 0.16  & 5 & 0     & 329.8 & 0     \\
		&   &     & 0.5  & 5 & 0.09  & 5 & 0     & 5 & 0.75  & 5 & 0     & 5 & 0.16  & 5 & 0     & 329.8 & 0     \\
		&   &     & 0.99 & 5 & 0.09  & 5 & 0     & 5 & 0.98  & 5 & 0     & 5 & 0.16  & 5 & 0     & 330   & 2.8   \\
		&   &     & 1    & 5 & 0.04  & 5 & 0     & 5 & 0.94  & 5 & 0     & 5 & 0.17  & 5 & 0     & 330   & 2.8   \\ \hline
		& 4 & 25  & 0    & 5 & 0.01  & 5 & 0     & 5 & 0.97  & 5 & 0     & 5 & 0.13  & 5 & 0     & 126.6 & 1     \\
		&   &     & 0.01 & 5 & 0.01  & 5 & 0     & 5 & 0.96  & 5 & 0     & 5 & 0.2   & 5 & 0     & 126.6 & 1     \\
		&   &     & 0.5  & 5 & 0.01  & 5 & 0     & 5 & 1.1   & 5 & 0     & 5 & 0.08  & 5 & 0     & 127.8 & 1.6   \\
		&   &     & 0.99 & 5 & 0.01  & 5 & 0     & 5 & 1.31  & 5 & 0     & 5 & 0.03  & 5 & 0     & 128.8 & 4     \\
		&   &     & 1    & 5 & 0.01  & 5 & 0     & 5 & 1.47  & 5 & 0     & 5 & 0.03  & 5 & 0     & 128.8 & 6.6   \\ \hline
		16 & 4 & 25  & 0    & 2 & 832.6 & 5 & 7.6   & 0 & -     & 5 & 27.8  & 3 & 327.8 & 5 & 7     & 679.4 & 7     \\
		&   &     & 0.01 & 5 & 134   & 5 & 0     & 0 & -     & 5 & 5.1   & 5 & 36.74 & 5 & 0     & 713.4 & 0     \\
		&   &     & 0.5  & 5 & 39.15 & 5 & 0     & 0 & -     & 5 & 213   & 5 & 28.86 & 5 & 0     & 718   & 4.6   \\
		&   &     & 0.99 & 5 & 3.29  & 5 & 0     & 0 & -     & 5 & 439.2 & 5 & 2.32  & 5 & 0     & 720.8 & 11.6  \\
		&   &     & 1    & 5 & 0.71  & 5 & 0     & 0 & -     & 5 & 413.8 & 5 & 2.03  & 5 & 0     & 720.8 & 11.6  \\ \hline
		24 & 2 & 25  & 0    & 0 & -     & 0 & -     & 4 & 92.82 & 5 & 19.8  & 5 & 309   & 5 & 0     & 3382  & 0     \\
		&   &     & 0.01 & 0 & -     & 0 & -     & 0 & -     & 5 & 4.11  & 4 & 585.2 & 4 & 0     & 3516  & 0     \\
		&   &     & 0.5  & 0 & -     & 0 & -     & 1 & 1561  & 5 & 135   & 4 & 418   & 4 & 0     & 3516  & 0     \\
		&   &     & 0.99 & 0 & -     & 0 & -     & 1 & 1706  & 5 & 218.6 & 4 & 428.6 & 4 & 0     & 3516  & 0     \\
		&   &     & 1    & 0 & -     & 0 & -     & 2 & 1459  & 5 & 215.2 & 4 & 432.6 & 4 & 0     & 3516  & 0     \\ \hline
		& 3 & 25  & 0    & 0 & -     & 0 & -     & 0 & -     & 5 & 113.4 & 0 & -     & 5 & 86.6  & 2186  & 86.6  \\
		&   &     & 0.01 & 0 & -     & 0 & -     & 0 & -     & 5 & 63.1  & 5 & 112.8 & 5 & 0     & 2347  & 0     \\
		&   &     & 0.5  & 0 & -     & 0 & -     & 0 & -     & 5 & 946.5 & 5 & 114.7 & 5 & 0     & 2347  & 0     \\
		&   &     & 0.99 & 0 & -     & 0 & -     & 0 & -     & 5 & 1829  & 5 & 115.7 & 5 & 0     & 2347  & 0     \\
		&   &     & 1    & 0 & -     & 0 & -     & 0 & -     & 5 & 1878  & 5 & 114.2 & 5 & 0     & 2347  & 0     \\ \hline
		& 4 & 25  & 0    & 0 & -     & 0 & -     & 0 & -     & 5 & 85    & 0 & -     & 5 & 72.2  & 1584  & 72.2  \\
		&   &     & 0.01 & 0 & -     & 0 & -     & 0 & -     & 5 & 65.9  & 5 & 61.93 & 5 & 0     & 1756  & 0     \\
		&   &     & 0.5  & 0 & -     & 0 & -     & 0 & -     & 5 & 698.3 & 5 & 56.21 & 5 & 0     & 1756  & 0     \\
		&   &     & 0.99 & 0 & -     & 0 & -     & 0 & -     & 5 & 1322  & 5 & 39.82 & 5 & 0     & 1757  & 6.6   \\
		&   &     & 1    & 5 & 131.2 & 5 & 0     & 0 & -     & 5 & 1336  & 5 & 45.35 & 5 & 0     & 1757  & 6.6   \\ \hline
		& 6 & 5   & 0    & 0 & -     & 5 & 7.4   & 0 & -     & 5 & 7.2   & 0 & -     & 5 & 8.4   & 197   & 8.4   \\
		&   &     & 0.01 & 0 & -     & 5 & 2.37  & 0 & -     & 5 & 8.66  & 0 & -     & 5 & 1.19  & 224.6 & 1.2   \\
		&   &     & 0.5  & 5 & 215.1 & 5 & 0     & 0 & -     & 5 & 79.4  & 5 & 601.9 & 5 & 0     & 229   & 3.2   \\
		&   &     & 0.99 & 5 & 12.81 & 5 & 0     & 0 & -     & 5 & 148.8 & 5 & 13.55 & 5 & 0     & 229.8 & 5.4   \\
		&   &     & 1    & 5 & 1.46  & 5 & 0     & 0 & -     & 5 & 151.4 & 5 & 8.43  & 5 & 0     & 229.8 & 6     \\ \hline
		&   & 25  & 0    & 0 & -     & 5 & 40    & 0 & -     & 5 & 52.4  & 0 & -     & 5 & 42.6  & 955   & 42.6  \\
		&   &     & 0.01 & 2 & 833.7 & 5 & 5.76  & 0 & -     & 5 & 44.14 & 5 & 543.8 & 5 & 0     & 1118  & 0     \\
		&   &     & 0.5  & 4 & 746.6 & 5 & 0.4   & 0 & -     & 5 & 415   & 5 & 184.2 & 5 & 0     & 1125  & 4.2   \\
		&   &     & 0.99 & 5 & 58.2  & 5 & 0     & 0 & -     & 5 & 756.8 & 5 & 27.27 & 5 & 0     & 1128  & 25.6  \\
		&   &     & 1    & 5 & 4.87  & 5 & 0     & 0 & -     & 5 & 772.6 & 5 & 15.77 & 5 & 0     & 1128  & 25.8  \\ \hline
		&   & 100 & 0    & 0 & -     & 5 & 239.4 & 0 & -     & 5 & 217.6 & 0 & -     & 5 & 178.2 & 3840  & 178.2 \\
		&   &     & 0.01 & 2 & 88.69 & 5 & 38.18 & 0 & -     & 5 & 197.2 & 5 & 241.4 & 5 & 0     & 4489  & 0     \\
		&   &     & 0.5  & 2 & 75.23 & 3 & 4.33  & 0 & -     & 5 & 1724  & 5 & 195.1 & 5 & 0     & 4489  & 0     \\
		&   &     & 0.99 & 5 & 121.4 & 5 & 0     & 0 & -     & 5 & 3195  & 5 & 41.02 & 5 & 0     & 4503  & 54.6  \\
		&   &     & 1    & 5 & 15.76 & 5 & 0     & 0 & -     & 5 & 3187  & 5 & 34.15 & 5 & 0     & 4503  & 58.6  \\ \hline
		32 & 4 & 25  & 0    & 0 & -     & 0 & -     & 0 & -     & 5 & 127.6 & 0 & -     & 5 & 127.8 & 2869  & 127.8 \\
		&   &     & 0.01 & 0 & -     & 0 & -     & 0 & -     & 5 & 128.6 & 3 & 921.6 & 5 & 0.02  & 3204  & 0     \\
		&   &     & 0.5  & 0 & -     & 0 & -     & 0 & -     & 5 & 1371  & 4 & 976.8 & 5 & 0.33  & 3205  & 0     \\
		&   &     & 0.99 & 0 & -     & 0 & -     & 0 & -     & 5 & 2610  & 5 & 980.2 & 5 & 0     & 3208  & 16.4  \\
		&   &     & 1    & 0 & -     & 0 & -     & 0 & -     & 5 & 2623  & 4 & 669.6 & 5 & 0.31  & 3208  & 23.4  \\ \hline
		& 8 & 25  & 0    & 0 & -     & 5 & 56.8  & 0 & -     & 5 & 63.6  & 0 & -     & 5 & 60    & 1228  & 60    \\
		&   &     & 0.01 & 0 & -     & 5 & 18.65 & 0 & -     & 5 & 67.79 & 0 & -     & 5 & 1.19  & 1530  & 1     \\
		&   &     & 0.5  & 1 & 1244  & 5 & 7.5   & 0 & -     & 5 & 568.8 & 2 & 868.6 & 5 & 1.57  & 1550  & 12.6  \\
		&   &     & 0.99 & 5 & 210   & 5 & 0     & 0 & -     & 5 & 1089  & 5 & 68.13 & 5 & 0     & 1554  & 20.6  \\
		&   &     & 1    & 5 & 15.06 & 5 & 0     & 0 & -     & 5 & 1084  & 5 & 55.95 & 5 & 0     & 1554  & 27   
	\end{tabular}
\end{table}

\begin{table}[]
	\centering
	\tiny
	\caption{Detailed solution statistics for \textbf{G4}}
	\label{tab:G4Details}
	\begin{tabular}{llll|llll|llll|llll|ll}
		&&&&
		\multicolumn{4}{c|}{\textbf{EXP}} &
		\multicolumn{4}{c|}{\textbf{BC}} &
		\multicolumn{4}{c|}{\textbf{BCP}} &
		\multicolumn{2}{c}{ }
		\\ 
		$n$ &
		$m$ &
		$\utility^{\mathrm{max}}$ &
		$\alpha$ &
		$n^s$ &
		$\overline{\textnormal{time}}$ &
		$n^g$
		&
		$\overline{\textnormal{gap}}$ &
		$n^s$ &
		$\overline{\textnormal{time}}$ &
		$n^g$
		&
		$\overline{\textnormal{gap}}$ &
		$n^s$ &
		$\overline{\textnormal{time}}$ &
		$n^g$
		&
		$\overline{\textnormal{gap}}$ & 
		$\utility^{\mathrm{best}}$ &
		$\uplift^{\mathrm{best}}$
		\\ \hline
		8  & 2 & 25  & 0    & 5 & 0.19  & 5 & 0     & 5 & 1.29  & 5 & 0     & 5 & 0.36  & 5 & 0     & 324.6 & 0     \\
		&   &     & 0.01 & 5 & 0.09  & 5 & 0     & 5 & 0.78  & 5 & 0     & 5 & 0.37  & 5 & 0     & 332.2 & 0     \\
		&   &     & 0.5  & 5 & 0.11  & 5 & 0     & 5 & 0.39  & 5 & 0     & 5 & 0.35  & 5 & 0     & 332.2 & 0     \\
		&   &     & 0.99 & 5 & 0.08  & 5 & 0     & 5 & 0.46  & 5 & 0     & 5 & 0.16  & 5 & 0     & 335.8 & 7.4   \\
		&   &     & 1    & 5 & 0.04  & 5 & 0     & 5 & 0.59  & 5 & 0     & 5 & 0.19  & 5 & 0     & 335.8 & 7.4   \\ \hline
		& 4 & 25  & 0    & 5 & 0.01  & 5 & 0     & 5 & 0.51  & 5 & 0     & 5 & 0.3   & 5 & 0     & 121.6 & 0.6   \\
		&   &     & 0.01 & 5 & 0.01  & 5 & 0     & 5 & 0.47  & 5 & 0     & 5 & 0.28  & 5 & 0     & 121.6 & 0.6   \\
		&   &     & 0.5  & 5 & 0.01  & 5 & 0     & 5 & 0.48  & 5 & 0     & 5 & 0.06  & 5 & 0     & 124.8 & 1.6   \\
		&   &     & 0.99 & 5 & 0.01  & 5 & 0     & 5 & 0.41  & 5 & 0     & 5 & 0.03  & 5 & 0     & 124.8 & 1.6   \\
		&   &     & 1    & 5 & 0.01  & 5 & 0     & 5 & 0.38  & 5 & 0     & 5 & 0.03  & 5 & 0     & 124.8 & 5.4   \\ \hline
		16 & 4 & 25  & 0    & 5 & 600.2 & 5 & 0     & 0 & -     & 5 & 35.4  & 0 & -     & 5 & 32.4  & 653   & 32.4  \\
		&   &     & 0.01 & 5 & 219.9 & 5 & 0     & 0 & -     & 5 & 12.19 & 5 & 413.6 & 5 & 0     & 704   & 0     \\
		&   &     & 0.5  & 5 & 38.45 & 5 & 0     & 0 & -     & 5 & 154.7 & 5 & 71.4  & 5 & 0     & 710.4 & 1.6   \\
		&   &     & 0.99 & 5 & 4.35  & 5 & 0     & 0 & -     & 5 & 283.3 & 5 & 3.35  & 5 & 0     & 723.2 & 21.8  \\
		&   &     & 1    & 5 & 0.64  & 5 & 0     & 0 & -     & 5 & 276.2 & 5 & 3.18  & 5 & 0     & 723.2 & 21.8  \\ \hline
		24 & 2 & 25  & 0    & 0 & -     & 0 & -     & 0 & -     & 5 & 224   & 1 & 383.4 & 5 & 148.2 & 3397  & 148.2 \\
		&   &     & 0.01 & 0 & -     & 0 & -     & 2 & 1344  & 5 & 5.56  & 4 & 388.8 & 4 & 0     & 3678  & 0     \\
		&   &     & 0.5  & 0 & -     & 0 & -     & 4 & 1068  & 5 & 28.4  & 5 & 372.5 & 5 & 0     & 3678  & 0     \\
		&   &     & 0.99 & 0 & -     & 0 & -     & 5 & 1024  & 5 & 0     & 4 & 324.8 & 4 & 0     & 3680  & 10.4  \\
		&   &     & 1    & 0 & -     & 0 & -     & 5 & 973.1 & 5 & 0     & 5 & 329.2 & 5 & 0     & 3680  & 10.4  \\ \hline
		& 3 & 25  & 0    & 0 & -     & 0 & -     & 0 & -     & 5 & 185.2 & 0 & -     & 5 & 143.2 & 2159  & 143.2 \\
		&   &     & 0.01 & 0 & -     & 0 & -     & 0 & -     & 5 & 108.3 & 4 & 225.2 & 5 & 0     & 2452  & 0     \\
		&   &     & 0.5  & 0 & -     & 0 & -     & 0 & -     & 5 & 806.1 & 4 & 219.9 & 5 & 0.12  & 2452  & 0     \\
		&   &     & 0.99 & 0 & -     & 0 & -     & 0 & -     & 5 & 1429  & 5 & 255.6 & 5 & 0     & 2462  & 35.6  \\
		&   &     & 1    & 0 & -     & 0 & -     & 0 & -     & 5 & 1536  & 5 & 79.09 & 5 & 0     & 2462  & 35.6  \\ \hline
		& 4 & 25  & 0    & 0 & -     & 0 & -     & 0 & -     & 5 & 115.4 & 0 & -     & 5 & 100   & 1557  & 100   \\
		&   &     & 0.01 & 0 & -     & 0 & -     & 0 & -     & 5 & 93.32 & 3 & 240.8 & 5 & 5.63  & 1777  & 5.6   \\
		&   &     & 0.5  & 0 & -     & 0 & -     & 0 & -     & 5 & 567.3 & 3 & 201.7 & 5 & 3.78  & 1785  & 6.8   \\
		&   &     & 0.99 & 0 & -     & 0 & -     & 0 & -     & 5 & 1019  & 5 & 64.81 & 5 & 0     & 1797  & 43.8  \\
		&   &     & 1    & 5 & 141.9 & 5 & 0     & 0 & -     & 5 & 1021  & 5 & 73.92 & 5 & 0     & 1797  & 51.8  \\ \hline
		& 6 & 5   & 0    & 0 & -     & 5 & 5.6   & 0 & -     & 5 & 8.2   & 0 & -     & 5 & 9.4   & 192.8 & 9.4   \\
		&   &     & 0.01 & 1 & 1657  & 5 & 3.15  & 0 & -     & 5 & 9.26  & 0 & -     & 5 & 2.77  & 221.4 & 2.8   \\
		&   &     & 0.5  & 5 & 168.3 & 5 & 0     & 0 & -     & 5 & 58.2  & 3 & 712.9 & 5 & 0.13  & 225.6 & 4.4   \\
		&   &     & 0.99 & 5 & 22.44 & 5 & 0     & 0 & -     & 5 & 111.4 & 5 & 25.75 & 5 & 0     & 225.6 & 4.4   \\
		&   &     & 1    & 5 & 1.55  & 5 & 0     & 0 & -     & 5 & 105.8 & 5 & 16.26 & 5 & 0     & 225.6 & 5.6   \\ \hline
		&   & 25  & 0    & 0 & -     & 5 & 53.8  & 0 & -     & 5 & 61.4  & 0 & -     & 5 & 52.2  & 959.8 & 52.2  \\
		&   &     & 0.01 & 0 & -     & 5 & 17.91 & 0 & -     & 5 & 59.87 & 0 & -     & 5 & 9.96  & 1087  & 9.8   \\
		&   &     & 0.5  & 1 & 1769  & 5 & 6.5   & 0 & -     & 5 & 345.1 & 0 & -     & 5 & 4.35  & 1120  & 21.2  \\
		&   &     & 0.99 & 5 & 39.38 & 5 & 0     & 0 & -     & 5 & 618.7 & 5 & 39.06 & 5 & 0     & 1125  & 33.4  \\
		&   &     & 1    & 5 & 4.94  & 5 & 0     & 0 & -     & 5 & 635.6 & 5 & 30.87 & 5 & 0     & 1125  & 36    \\ \hline
		&   & 100 & 0    & 0 & -     & 5 & 258   & 0 & -     & 5 & 254.6 & 0 & -     & 5 & 204.4 & 3955  & 204.4 \\
		&   &     & 0.01 & 0 & -     & 5 & 89.05 & 0 & -     & 5 & 237.3 & 0 & -     & 5 & 6.59  & 4456  & 6     \\
		&   &     & 0.5  & 0 & -     & 4 & 35.75 & 0 & -     & 5 & 1397  & 1 & 1232  & 5 & 11.38 & 4507  & 27.2  \\
		&   &     & 0.99 & 5 & 124.3 & 5 & 0     & 0 & -     & 5 & 2506  & 5 & 88.97 & 5 & 0     & 4540  & 105.2 \\
		&   &     & 1    & 5 & 14.06 & 5 & 0     & 0 & -     & 5 & 2563  & 5 & 106.9 & 5 & 0     & 4540  & 105.2 \\ \hline
		32 & 4 & 25  & 0    & 0 & -     & 0 & -     & 0 & -     & 5 & 223.6 & 0 & -     & 5 & 182.4 & 2908  & 182.4 \\
		&   &     & 0.01 & 0 & -     & 0 & -     & 0 & -     & 5 & 195.6 & 1 & 914.8 & 5 & 46.71 & 3329  & 47    \\
		&   &     & 0.5  & 0 & -     & 0 & -     & 0 & -     & 5 & 1238  & 1 & 1423  & 5 & 25.34 & 3338  & 42.6  \\
		&   &     & 0.99 & 0 & -     & 0 & -     & 0 & -     & 5 & 2317  & 5 & 437.2 & 5 & 0     & 3356  & 95.8  \\
		&   &     & 1    & 0 & -     & 0 & -     & 0 & -     & 5 & 2316  & 5 & 508.4 & 5 & 0     & 3356  & 95.8  \\ \hline
		& 8 & 25  & 0    & 0 & -     & 5 & 62.6  & 0 & -     & 5 & 67.6  & 0 & -     & 5 & 62.2  & 1256  & 62.2  \\
		&   &     & 0.01 & 0 & -     & 5 & 23.44 & 0 & -     & 5 & 72.39 & 0 & -     & 5 & 13.42 & 1482  & 13.2  \\
		&   &     & 0.5  & 0 & -     & 5 & 15.2  & 0 & -     & 5 & 519.8 & 0 & -     & 5 & 8.75  & 1515  & 20.4  \\
		&   &     & 0.99 & 5 & 216.1 & 5 & 0     & 0 & -     & 5 & 918   & 5 & 113.8 & 5 & 0     & 1519  & 32.4  \\
		&   &     & 1    & 5 & 15.17 & 5 & 0     & 0 & -     & 5 & 943.6 & 5 & 110.5 & 5 & 0     & 1519  & 33.8 
	\end{tabular}
\end{table}

\begin{landscape}
	\begin{table}[h!]
		\centering
		\tiny
		\caption{Comparing solutions with and without characteristic constraints}
		\label{tab:CharCons}
		\begin{tabular}{lll|llll|llll|llll|llll}
			&&&
			\multicolumn{4}{c|}{\textbf{G1}} &
			\multicolumn{4}{c|}{\textbf{G2}} &
			\multicolumn{4}{c|}{\textbf{G3}} &
			\multicolumn{4}{c}{\textbf{G4}}
			\\
			&&&
			\multicolumn{2}{c}{W/out Cons} &
			\multicolumn{2}{c|}{W/ Cons} &
			\multicolumn{2}{c}{W/out Cons} &
			\multicolumn{2}{c|}{W/ Cons} &
			\multicolumn{2}{c}{W/out Cons} &
			\multicolumn{2}{c|}{W/ Cons} &
			\multicolumn{2}{c}{W/out Cons} &
			\multicolumn{2}{c}{W/ Cons} 
			\\
			$n$ &
			$m$ &
			$\alpha$ &
			$\bar{r}$ &
			$\bar{u}$ &
			$\bar{r}$ &
			$\bar{u}$ &
			$\bar{r}$ &
			$\bar{u}$ &
			$\bar{r}$ &
			$\bar{u}$ &
			$\bar{r}$ &
			$\bar{u}$ &
			$\bar{r}$ &
			$\bar{u}$ &
			$\bar{r}$ &
			$\bar{u}$ &
			$\bar{r}$ &
			$\bar{u}$
			\\ \hline \hline
			8  & 2 & 0.01 & 0     & 304.2  & 4.6   & 303.6  & 0     & 305    & 11.4  & 304.2  & 0    & 329.8  & 0    & 324.2  & 0    & 332.2  & 3    & 315.2  \\
			&   & 0.5  & 0     & 304.2  & 4.6   & 303.6  & 0     & 305    & 11.6  & 305    & 0    & 329.8  & 0    & 324.2  & 0    & 332.2  & 3    & 315.2  \\
			&   & 0.99 & 36.6  & 319    & 41.4  & 316    & 29.6  & 315.8  & 37    & 312.4  & 2.8  & 330    & 5.2  & 325.8  & 7.4  & 335.8  & 8.6  & 316    \\ \hline
			& 4 & 0.01 & 0     & 109    & 0     & 100.2  & 2     & 106.6  & 0     & 99.2   & 1    & 126.6  & 0    & 110.8  & 0.6  & 121.6  & 0    & 92.6   \\
			&   & 0.5  & 1     & 110.6  & 0     & 100.2  & 2.4   & 110.4  & 0     & 99.2   & 1.6  & 127.8  & 0    & 110.8  & 1.6  & 124.8  & 0.2  & 93     \\
			&   & 0.99 & 9.8   & 113.4  & 3.6   & 102.2  & 8.6   & 113.2  & 0     & 99.2   & 4    & 128.8  & 0.8  & 111    & 1.6  & 124.8  & 0.2  & 93     \\ \hline
			16 & 4 & 0.01 & 3.4   & 633.4  & 13.2  & 630.2  & 6.6   & 636.2  & 2.4   & 639.6  & 0    & 713.4  & 0    & 709    & 0    & 704    & 4.2  & 663.8  \\
			&   & 0.5  & 12.8  & 651.8  & 19.2  & 642.8  & 14    & 652.8  & 19.4  & 656.8  & 4.6  & 718    & 1.8  & 711.4  & 1.6  & 710.4  & 8.6  & 674.6  \\
			&   & 0.99 & 47    & 666.6  & 52.2  & 657.8  & 44.4  & 668    & 38.2  & 666.2  & 11.6 & 720.8  & 12.4 & 714.6  & 21.8 & 723.2  & 23.2 & 679.8  \\ \hline
			24 & 2 & 0.01 & 0     & 3312.8 & 303.8 & 3402.4 & 0     & 3305.2 & 314.6 & 3400.8 & 0    & 3516   & 0    & 3516   & 0    & 3677.6 & 0    & 3668.8 \\
			&   & 0.5  & 0     & 3312.8 & 254.4 & 3402.6 & 0     & 3305.2 & 308   & 3405.6 & 0    & 3516   & 0    & 3516   & 0    & 3677.6 & 12.4 & 3672.8 \\
			&   & 0.99 & 370.4 & 3428.8 & 370.4 & 3428.8 & 369.6 & 3418.6 & 369.6 & 3418.6 & 0    & 3516   & 0    & 3516   & 10.4 & 3680.4 & 22.8 & 3675.6 \\ \hline
			& 3 & 0.01 & 0     & 2094.2 & 156.4 & 2207.4 & 18.4  & 2095.8 & 141.4 & 2214.6 & 0    & 2347.2 & 0    & 2347.2 & 0    & 2452.2 & 10.4 & 2385.2 \\
			&   & 0.5  & 0     & 2094.2 & 150.8 & 2220.2 & 27.6  & 2124.4 & 150.4 & 2229.2 & 0    & 2347.2 & 0    & 2347.2 & 0    & 2452.2 & 10.6 & 2388   \\
			&   & 0.99 & 209.2 & 2245.6 & 209.2 & 2245.6 & 202.6 & 2246.8 & 202.6 & 2246.8 & 0    & 2347.2 & 0    & 2347.2 & 35.6 & 2461.6 & 72.6 & 2404.6 \\ \hline
			& 4 & 0.01 & 44.8  & 1509.6 & 81.6  & 1610.8 & 43.6  & 1525.4 & 73.4  & 1617.4 & 0    & 1755.8 & 0    & 1755.4 & 5.6  & 1777   & 2    & 1732   \\
			&   & 0.5  & 75.8  & 1624.6 & 75.6  & 1621.2 & 74.8  & 1622.4 & 74.6  & 1624.6 & 0    & 1755.8 & 0    & 1755.4 & 6.8  & 1784.8 & 2    & 1732   \\
			&   & 0.99 & 126.6 & 1636.8 & 130.6 & 1635.6 & 118   & 1641.6 & 118   & 1641.6 & 6.6  & 1757.2 & 10.8 & 1757.2 & 43.8 & 1797.4 & 50   & 1747.8 \\ \hline
			& 6 & 0.01 & 20.4  & 999    & 27.2  & 995.2  & 22.2  & 969.4  & 24    & 985.8  & 0    & 1118.2 & 0    & 1102.6 & 9.8  & 1087.2 & 18.6 & 1034.6 \\
			&   & 0.5  & 28.4  & 1009.8 & 29.2  & 1004   & 37.4  & 1010   & 33.8  & 1002.2 & 4.2  & 1124.8 & 4.2  & 1107.4 & 21.2 & 1120   & 24   & 1046.4 \\
			&   & 0.99 & 56.2  & 1021.8 & 58.8  & 1016.2 & 58.4  & 1017.4 & 57.4  & 1011   & 25.6 & 1128.2 & 18.8 & 1113.8 & 33.4 & 1125.2 & 27.8 & 1047.4 \\ \hline
			32 & 4 & 0.01 & 66.4  & 2769.4 & 210.2 & 3031.4 & 21.6  & 2792.4 & 220.4 & 3002   & 0    & 3204.4 & 0    & 3205.2 & 47   & 3329.2 & 44.6 & 3237.2 \\
			&   & 0.5  & 171.4 & 2984   & 216.2 & 3034.4 & 172   & 2963.4 & 214.2 & 3011.2 & 0    & 3205.2 & 0    & 3205.2 & 42.6 & 3338.2 & 62.2 & 3271   \\
			&   & 0.99 & 240.8 & 3028.4 & 228.6 & 3036.2 & 247.4 & 3010.6 & 243.6 & 3019.6 & 16.4 & 3207.6 & 16.4 & 3207.6 & 95.8 & 3356.2 & 116  & 3286.8 \\ \hline
			& 8 & 0.01 & 27    & 1298   & 33.2  & 1323.4 & 21.8  & 1317.4 & 28.2  & 1348.4 & 1    & 1529.6 & 1.4  & 1504.8 & 13.2 & 1481.8 & 18.6 & 1409.2 \\
			&   & 0.5  & 45.8  & 1377.8 & 42.4  & 1366   & 34.8  & 1376.2 & 37.2  & 1376.6 & 12.6 & 1550.4 & 4.8  & 1528   & 20.4 & 1515.2 & 21.4 & 1419.4 \\
			&   & 0.99 & 64.4  & 1385.4 & 64.8  & 1377.8 & 64    & 1389.4 & 55    & 1384   & 20.6 & 1554   & 23   & 1535.4 & 32.4 & 1519.4 & 29.8 & 1422.2
		\end{tabular}
	\end{table}
\end{landscape}

\end{document}